\documentclass[11p,a4paper]{article}
\usepackage{amsfonts}
\usepackage{amssymb}
\usepackage{amsthm}
\usepackage{amsmath}
\usepackage{graphicx}
\usepackage{mathrsfs}
\usepackage{graphics}
\usepackage{abstract}
\usepackage{color}
\usepackage{cite}
\usepackage[colorlinks=true,citecolor=red]{hyperref}
\hypersetup{linkcolor=blue}
\graphicspath{ {./result/} }
\usepackage{titling}
\usepackage{authblk}
\usepackage{geometry}
\usepackage{fancyhdr}
\usepackage{lineno}
\usepackage{titlesec}
\titleformat{\section}[block]{\centering\Large\bfseries}{\thesection}{1em}{}
\titleformat{\subsection}[block]{\centering\large\bfseries}{\thesubsection}{1em}{}
\titleformat{\subsubsection}[block]{\centering\normalsize\bfseries}{\thesubsubsection}{1em}{}
\newtheorem{theorem}{\textbf{Theorem}}[section]
\newtheorem{lemma}{\textbf{Lemma}}[section]
\newtheorem{proposition}{\textbf{Proposition}}[section]
\newtheorem{corollary}{\textbf{Corollary}}[section]
\newtheorem{remark}{\textbf{Remark}}[section]
\newtheorem{definition}{\textbf{Definition}}[section]

\providecommand{\keywords}[1]{\textbf{\textit{Keywords---}} #1}
\allowdisplaybreaks[4]

\def\be{\begin{equation}}
\def\ee{\end{equation}}
\def\bea{\begin{eqnarray}}
\def\eea{\end{eqnarray}}
\def\bt{\begin{theorem}}
\def\et{\end{theorem}}
\def\bl{\begin{lemma}}
\def\el{\end{lemma}}
\def\br{\begin{remark}}
\def\er{\end{remark}}
\def\bp{\begin{proposition}}
\def\ep{\end{proposition}}
\def\bc{\begin{corollary}}
\def\ec{\end{corollary}}
\def\bd{\begin{definition}}
\def\ed{\end{definition}}

\def\Pi{\mathbf{\psi}}

\def\R{{\mathbb R}}

\usepackage{tikz}
\usetikzlibrary{calc,arrows,intersections}
\usepackage{enumitem} % 加载 enumitem 包
\setlist[enumerate]{leftmargin=24pt, itemsep=0pt, parsep=0pt, topsep=0pt, partopsep=0pt,  label=\bf{\textcolor{blue}{(\arabic*)}}}
\title{\bf{On the instability and stability of 
 non-homogeneous fluid in a bounded domain under the influence of a general potential}}

\author{
Liang Li$^{a}$
\thanks{llbohou@gzhu.edu.cn}
, Tao Tan$^{b}$\thanks{ tantao1@stu.scu.edu.cn}
and
Quan Wang$^{b}$
\thanks{Corresponding author:xihujunzi@scu.edu.cn }
\\ \footnotesize $^a$ School of Mathematics and Information Science,
\footnotesize Guangzhou University,
\\\footnotesize Guangzhou,
 Guangdong, 510000, China
  \\ \footnotesize $^{b}$ College of Mathematics, Sichuan University,
  \footnotesize
 Chengdu, Sichuan, 610065,  China
}
\date{}

\begin{document}
%\linenumbers
\maketitle

\begin{abstract}
We investigate the instability and stability of specific steady-state solutions of the two-dimensional non-homogeneous, incompressible, and viscous Navier-Stokes equations under the influence of a general potential $f$. This potential is commonly used to model fluid motions in celestial bodies. First, we demonstrate that the system admits only steady-state solutions of the form \(\left(\rho,\mathbf{V},p\right)=\left(\rho_{0},\mathbf{0},P_{0}\right)\), where $P_0$ and $\rho_0$
satisfy  the hydrostatic balance condition
$\nabla P_{0}=-\rho_{0}\nabla f$.
Additionally, the relationship between $\rho_0$ and the potential function $f$  is constrained by the condition
 $\left(\partial_{y}\rho_{0},-\partial_{x}\rho_{0}\right)\cdot\left(\partial_{x}f,\partial_{y}f\right)=0$, which allows us to express
 $\nabla\rho_{0}$ as $h\left(x,y\right)\nabla f$.
Second, when there exists a point $\left(x_{0},y_{0}\right)$ such that $h\left(x_{0},y_{0}\right)>0$,  we establish the linear instability of these solutions. Furthermore, we demonstrate their nonlinear instability in both the Lipschitz and Hadamard senses through detailed nonlinear energy estimates. This instability aligns with the well-known Rayleigh-Taylor instability.  Our study signficantly extends and generalizes the existing mathematical results, which have predominantly focused on the scenarios involving a uniform gravitational field characterized by $\nabla f=(0,g)$. Finally, we show that these steady states are linearly stable provided that $h\left(x,y\right)<0$ holds throughout the domain. Moreover, they exhibit nonlinear stability when $h\left(x,y\right)$ is a negative constant.
\end{abstract}

\newpage
\tableofcontents

\section{Introduction}
Fluid stability is a key research focus within nonlinear sciences. The main goal of investigating fluid stability is to improve our understanding of fluid flow dynamics and to devise methodologies and tools for effectively controlling and manipulating these flows in practical applications.  The extensive literature \cite{Friedlander2006,Friedlander2009,Galdi2022,Feng2023,Ji2023}
has been dedicated to the mathematical analysis of fluid stability and instability, highlighting its significance in both theoretical and applied contexts.

In this paper, we investigate the Rayleigh-Taylor (R-T) instability and stability of the nonhomogeneous incompressible and viscous Navier-Stokes equations (NIVNSE) within a bounded smooth domain  $\Omega$, subject to a general potential $f$.
The NIVNSE system reads
\begin{align}\label{moxing1121}
\begin{cases}
\partial_{t}\rho+\mathbf{V}\cdot\nabla\rho=0,
\\
\rho \partial_{t}\mathbf{V}+\rho\mathbf{V}\cdot\nabla\mathbf{V}+\nabla P=\mu\Delta\mathbf{V}-\rho \nabla f,
\\
\partial_{x}V_{1}+\partial_{y}V_{2}=0,
\end{cases}
\end{align}
where $\mathbf{V}$ is the velocity field, $\rho$ denotes the density, 
$P$ denotes the pressure, and \(f\) is a smooth potential function that can be employed to model the gravitational potential and other external forces.
In this study, the system \eqref{moxing1121} is subject to the following non-slip boundary condition:\begin{align}\label{wuguanghua0202}
\mathbf{V}|_{\partial\Omega}=\mathbf{0}.
\end{align}

The R-T instability, initially described by Rayleigh \cite{Rayleigh1882,Rayleigh1900} in the context of cirrus cloud formation and later by Taylor \cite{Taylor1950} for the scenario where a lighter fluid is positioned above a heavier fluid under downward acceleration,  
is a classic phenomenon in fluid dynamics. This hydrodynamic instability manifests not only at the interfacial boundary between immiscible fluids exhibiting significant density contrast, particularly when a denser fluid is superimposed above a less dense medium under external acceleration fields such as gravitational or inertial forces, but also within a single fluid when the derivative of its density is positive at a certain horizontal layer. Its importance extends across a wide range of disciplines, including astrophysics (e.g., supernova evolution \cite{Gull1975,Ribeyre2004}), geophysics (e.g., volcanic island formation, salt domes \cite{Wilcock1991,Mazariegos1996}, and lithospheric convection \cite{Chen2015}), nuclear physics(e.g., inertial confinement fusion \cite{Srinivasan2012}), and even industrial processes such as combustion \cite{Veynante2002} and underwater explosions \cite{Geers2002}. Beyond science,the R-T instability inspires artistic patterns in oil painting \cite{Delacalleja2014,Zetina2014}. For a detailed review of its applications, the reader is referred to \cite{Zhou2017,Zhou20171,Zhou2021}.

The NIVNSE system is a core model for describing fluid motion, with significant theoretical importance and practical application value. Thus, since its inception, it has attracted the research interest of numerous scholars \cite{Simon1978,Simon1990,Lions1996,Danchin2009} regarding well-posedness. For the well-posedness of \eqref{moxing1121}-\eqref{wuguanghua0202}, one can refer to \cite{hkim1987,Choe2003}. We simply list it as described in Lemma \ref{shidingxingwenti0202}. 
Given their wide applicability in science and engineering, there is extensive research on the R-T instability of the NIVNSE system. Jiang \cite{Jiang2013} investigated the nonlinear instability in the Lipschitz sense of the NIVNSE system in \(\mathbf{R}^{3}\) under a uniform gravitational field. Another study by Jiang \cite{Jiang2014} examined the nonlinear instability in Hadamard sense of some steady states of a three-dimensional NIVNSE system in a bounded domain of class \(C^{2}\) with Dirichlet-type boundary conditions. For more research on the R-T instability, one can consult the literature \cite{zhou2016,Livescu2021,Sonf2021,Cheung2023,Hwang2003,Guo2010,Jiang2019,Jiang20191,Peng2022,Mao2024,Xing2024} and the references therein. 

It is worth noting that the previous mathematical research \cite{Jiang2013,Jiang2014,zhou2016,Livescu2021,Sonf2021,Cheung2023,Hwang2003,Guo2010,Jiang2019,Jiang20191,Peng2022,Mao2024,Xing2024} on the R-T instability of the NIVNSE system have predominantly focused on the scenarios involving a uniform
gravitational field characterized by \(\nabla f=\left(0,g\right)\). 
However, in the context of geophysical and astrophysical fluid dynamics, when modeling fluid motions using the NIVNSE system, it is more physically accurate to replace the gravity term  $-\rho\left(0,g\right)^T$ by
$-\rho\nabla f$, where
$f$ is a potential function. This modification offers several advantages.
First, it better captures the effects of spatially varying gravitational forces. Second, it also accounts for other physical processes involved in fluid motions, such as chemical reactions of matter and radiation processes \cite{Michael2025}. Consequently, this model \eqref{moxing1121} is more amenable to incorporating these complex physical processes through extensions, providing a more accurate representation of the actual conditions in geophysical and astrophysical fluid dynamics.

\subsection{Exact steady states}
From a fluid mechanics perspective, steady-state solutions of hydrodynamic systems fundamentally characterize their equilibrium configurations. A rigorous mathematical characterization of these precise stationary states constitutes the essential foundation for investigating stability properties - determining whether infinitesimal perturbations will decay (stability) or grow exponentially (instability) - in equilibrated systems. Thus, at the beginning, we embark on an in-depth exploration of the steady-state solutions of \eqref{moxing1121}-\eqref{wuguanghua0202} and consider the following set of equations:
\begin{align}\label{wentifangcheng0201}
\begin{cases}
\mathbf{V}\cdot\nabla\rho=0,
\\
\rho\mathbf{V}\cdot\nabla\mathbf{V}+\nabla P=\mu\Delta\mathbf{V}-\rho\nabla f,
\\
\partial_{x}V_{1}+\partial_{y}V_{2}=0,~\mathbf{V}|_{\partial\Omega}=\mathbf{0}.
\end{cases}
\end{align}
Performing a careful analysis of \eqref{wentifangcheng0201},
we have derived the following conclusion.
\begin{lemma}\label{wentaijie0201}
Let \(\left(\mathbf{u},P_{0},\rho_{0}\right)\) be a classical solution of the hydrodynamic equilibrium equations \eqref{wentifangcheng0201}. This solution exhibits the following fundamental properties:
\begin{align}
\mathbf{u}=\mathbf{0},~\nabla P_{0}=-\rho_{0}\nabla f.
\end{align}
Additionally, the relationship between $\rho_0$ and $f$  is constrained by the condition
\begin{align}\label{wutian0201}
\left(\partial_{y}\rho_{0},-\partial_{x}\rho_{0}\right)\cdot 
\left(\partial_{x}f,\partial_{y}f\right)=0.
\end{align}
\end{lemma}
The proof of Lemma \ref{wentaijie0201} is provided in Appendix \ref{appendix0304}.
\begin{remark}
Lemma \ref{wentaijie0201} says that any 
stationary state of the system \eqref{moxing1121}-\eqref{wuguanghua0202}
is a state of hydrostatic equilibrium. And, the relationship \eqref{wutian0201} allow us to express
\begin{align}\label{wending0202}
\nabla\rho_{0}\left(x,y\right)=h\left(x,y\right)\nabla f\left(x,y\right),
\end{align}
where \(h=h\left(x,y\right)\) is a bounded and smooth function.
\end{remark}

We are interested in the instability and stability of 
 steady-state solutions satisfying \eqref{wending0202}. Physically,
 the density function must be a positive-valued function. We then assume
\begin{align}\label{xuyao0207}
\rho_{0}\left(x,y\right)>0~\text{for~ any~}\left(x,y\right)\in \Omega.
\end{align}
Careful analysis shows that the instability and stability of these steady-state solutions can be determined by the function \(h\). Specifically, we have the following findings: 
\begin{enumerate}
    \item \(\left(\mathbf{0},P_{0},\rho_{0}\right)\) is linearly and nonlinearly unstable if 
\(h|_{(x,y)=(x_{0},y_{0})}>0\), see \autoref{xianxingbuwending1221}-\ref{hadamardyiyixia0202};
\item \(\left(\mathbf{0},P_{0},\rho_{0}\right)\) is linearly stable if \(h(x,y)<0\) , see the first conclusion of  \autoref{wendingxing0225}; 
\item \(\left(\mathbf{0},P_{0},\rho_{0}\right)\) is nonlinearly stable when \(h\equiv~\text{const}<0\), see \autoref{wendingxing0225}.
\end{enumerate}
These findings provide a clear criterion for assessing the instability of the steady-state solutions based on the behavior of the function $h$. 
Our results significantly extend the previous results  \cite{Jiang2013,Jiang2014,zhou2016,Livescu2021,Sonf2021,Cheung2023,Hwang2003,Guo2010,Jiang2019,Jiang20191,Peng2022,Mao2024,Xing2024} , which have predominantly focused on scenarios involving a uniform gravitational field $\nabla f=(0,g)$.

\subsection{Linear and nonlinear instability}

To consider the linear and nonlinear instability of the 
steady-state solutions  \(\left(\mathbf{0},P_{0},\rho_{0}\right)\) 
satisfying \eqref{wending0202}-\eqref{xuyao0207}, we let their perturbation be
$(\widetilde{\mathbf{V}},\widetilde{P},\widetilde{\rho})$, i.e.,
\begin{align}\label{mianhuicai1219}
\rho=\widetilde{\rho}+\rho_{0},~
\mathbf{V}=\widetilde{\mathbf{V}},~
P=\widetilde{P}+P_{0}.
\end{align}
 Substitute \eqref{mianhuicai1219} into \eqref{moxing1121}, we obtain 
\begin{align*}
\begin{cases}
\widetilde{\rho}_{t}+\widetilde{\mathbf{V}}\cdot\nabla\left(\widetilde{\rho}+\rho_{0}\right)=0,
\\
\left(\widetilde{\rho}+\rho_{0}\right)\widetilde{\mathbf{V}}_{t}+\left(\widetilde{\rho}+\rho_{0}\right)\mathbf{\widetilde{V}}\cdot\nabla\widetilde{\mathbf{V}}
+\nabla\widetilde{P}+\nabla P_{0}=\mu\Delta\widetilde{\mathbf{V}}-\widetilde{\rho}\nabla f-\rho_{0}\nabla f,
\\
\nabla\cdot \widetilde{\mathbf{V}}=0,~
\widetilde{\mathbf{V}}|_{\partial\Omega}=\mathbf{0}.
\end{cases}
\end{align*}
Remove the tilde above and use \eqref{wending0202}, we get 
\begin{align}\label{raodong1219}
\begin{cases}
\rho_{t}+\mathbf{V}\cdot\nabla\rho=-h\mathbf{V}\cdot\nabla f,
\\
\left(\rho+\rho_{0}\right)\mathbf{V}_{t}+\left(\rho+\rho_{0}\right)\mathbf{V}\cdot\nabla\mathbf{V}
+\nabla P=\mu\Delta\mathbf{V}-\rho\nabla f,
\\
\nabla\cdot \mathbf{V}=0,~
\mathbf{V}|_{\partial\Omega}=\mathbf{0}.
\end{cases}
\end{align}
To complete the statement of the perturbed problem, we specify the initial data of \eqref{raodong1219} as 
\begin{align*}
\left(\mathbf{V},\rho\right)|_{t=0}
=\left(\mathbf{u}_{0},\theta_{0}\right),
\end{align*}
where \(\nabla\cdot\mathbf{u}_{0}=0\) and \(\mathbf{u}_{0}|_{\partial\Omega}=\mathbf{0}\). The resulting linearized equations are 
\begin{align}\label{xianxing1219}
\begin{cases}
\rho_{t}+h\mathbf{V}\cdot\nabla f=0,
\\
\rho_{0}\mathbf{V}_{t}+\nabla P=\mu\Delta\mathbf{V}-\rho\nabla f,
\\
\nabla\cdot\mathbf{V}=0,~\mathbf{V}|_{\partial\Omega}=\mathbf{0}.
\end{cases}
\end{align}

The first conclusion pertains to the linear instability:
\begin{theorem}\label{xianxingbuwending1221}[{\bf Linear instability}] If there exists a point \(\left(x_{0},y_{0}\right)\in \Omega\) such that \(h(x_{0},y_{0})>0\), 
then there exists a smooth initial data \(\left(\mathbf{V},\rho\right)|_{t=0}=\left(\mathbf{u}_{0},\theta_{0}\right)\) and \(\Lambda>0\) such that
\(\left(\mathbf{V},\rho\right)=e^{\Lambda t}\left(\mathbf{u}_{0},\theta_{0}\right)\) is the solution of the linearized system \eqref{xianxing1219}. Additionally, for any \(\mathbf{u}\in \mathbf{H}_{\text{div}}=\left\{\mathbf{u}\in \left[H_{0}^1\left(\Omega\right)\right]^2|\nabla\cdot\mathbf{u}=0\right\}\), we have
\begin{align*}
\Lambda^2 \left\|\sqrt{\rho_{0}}\mathbf{u}\right\|_{L^2\left(\Omega\right)}^2 
\geq -\Lambda\mu\left\|\nabla\mathbf{u}\right\|_{L^2\left(\Omega\right)}^2+\int_{\Omega}h\left|\mathbf{u}\cdot\nabla f\right|^2dxdy,
\end{align*}
and
\begin{align*}
\Lambda^2\int_{\Omega}\rho_{0}\mathbf{u}_{0}^2dxdy=
-\Lambda\mu\left\|\nabla\mathbf{u}_{0}\right\|_{L^2\left(\Omega\right)}^2+\int_{\Omega}h\left|\mathbf{u}_{0}\cdot\nabla f\right|^2dxdy,~\theta_{0}=-\frac{\mathbf{u}_{0}\cdot\nabla\rho_{0}}{\Lambda}.
\end{align*}
\end{theorem}
We employ the method of separation of spatiotemporal variables to seek the solutions that exhibit exponential growth, specifically of the form \(e^
{\Lambda t}\) with $\Lambda>0$.  This approach eliminates the variable $t$, thereby reducing equations \eqref{xianxing1219} to equations \eqref{wanzhe1219}. The latter is then solved using the modified variational method proposed in \cite{Guo2010}. This process also establishes the existence of solutions \eqref{wanzhe1219} and determines the optimal exponential growth rate $\Lambda$
for velocity field $\mathbf{u}$. For detailed derivations and results, refer to Propositions \ref{pingfan1220}--\ref{xiangdeng1221}.

The subsequent content is  about the nonlinear instability in Lipschitz and Hadamard senses:
\begin{theorem}\label{lipuxizhi0202}[{\bf Nonlinear instability in Lipschitz sense}] Assume that \(\inf\limits_{\left(x,y\right)\in\Omega}\) \(\left\{\rho_{0}+\theta_{0}\right\}\)\(>\sigma>0\),
the steady state \(\left(\mathbf{0},\rho_{0}\right)\) of the system \eqref{moxing1121}-\eqref{wuguanghua0202} satisfies \eqref{wending0202}-\eqref{xuyao0207} and there exists a point \(\left(x_{0},y_{0}\right)\in \Omega\) such that \(h\left(x_{0},y_{0}\right)>0\). Then, the steady state \(\left(\mathbf{0},\rho_{0}\right)\) is unstable in Lipschitz sense. More precisely, for any \(K>0\), \(\delta_{0}>0\) sufficiently small and any Lispchitz continuous functions \(F\) satisfying $
F\left(y\right)\leq Ky$, and $y\in [0,\infty)$,
there exists an initial data 
\[\left(\mathbf{u}_{0},\theta_{0}\right)\in\left[H^{2}\left(\Omega\right)\right]^{2}\times \left[H^{1}\left(\Omega\right)\cap L^{\infty}\left(\Omega\right)\right]
\]
satisfying \(\left\|\left(\mathbf{u}_{0},\theta_{0}\right)\right\|_{H^2\left(\Omega\right)}:=\sqrt{\left\|\mathbf{u}_{0}\right\|_{H^2\left(\Omega\right)}^2+\left\|\theta_{0}\right\|_{H^{1}\left(\Omega\right)}^2}\leq \delta_{0}\)
such that the unique strong solution \(\left(\mathbf{V},\rho\right)\) of the nonlinear problem \eqref{raodong1219} with the initial data \(\left(\mathbf{u}_{0},\theta_{0}\right)\) satisfies 
\[
\left\|V_{i}\left(t_{K_{i}}\right)\right\|_{L^2\left(\Omega\right)}>F\left(\left\|\left(\mathbf{u}_{0},\theta_{0}\right)\right\|_{H^2\left(\Omega\right)}\right),~t_{K_{i}}=\frac{1}{\Lambda}\ln{\frac{2K}{\tau_{0i}}}\in (0,T_{max}),
\]
where \(\tau_{0i}:=\frac{\left\|u_{0i}\right\|_{L^2\left(\Omega\right)}}{\left\|\left(\mathbf{u}_{0},\theta_{0}\right)\right\|_{H^2\left(\Omega\right)}}\)\(\left(i=1,2\right)\) and the constant \(\Lambda\) is defined in Proposition \ref{xiangdeng1221} and \(T_{max}\) denotes the maximal existence time of the solution \(\left(\mathbf{V},\rho\right)\).
\end{theorem}
The detailed proof of \autoref{lipuxizhi0202} is given in Section 3.1.
Our approach to the proof can be outlined in four key steps as follows:
\begin{enumerate}
    \item  Verify the uniqueness of the weak solution of the linearized perturbed problem \eqref{xianxing1219} with initial data \(\left(\mathbf{u}_{0},\theta_{0}\right)\), as detailed in Lemma \ref{weiyixing1230}.
\item Establish the estimate of the strong solution of nonlinear problem \eqref{raodong1219}, as presented in Proposition \ref{nengliangguji0102}.
\item Using the above estimate, construct a family of solutions for
the nonlinear problem \eqref{raodong1219}. Verify that the limit of this family of solutions corresponds to the solution
of a linearized problem \eqref{xianxing1219}, using the lemma \ref{weiyixing1230}.
\item Using the exponential growth rate of the solution to the linearized problem and the estimate (see Proposition \ref{nengliangguji0102}) to derive a contradiction, as shown in
Lemma \ref{yinli0103}.
\end{enumerate}
\begin{theorem}\label{hadamardyiyixia0202}[{\bf Nonlinear instability in Hadamard sense}]
Under the conditions of Theorem \ref{lipuxizhi0202},
the steady-state solution \(\left(\mathbf{0},\rho_{0}\right)\) is unstable in Hadamard sense. That is, there exist two constants  \(\epsilon\) and \(\delta_{0}\), and functions \(\left(\mathbf{u}_{0},\theta_{0}\right)\)\(\in \left[H^{2}\left(\Omega\right)\right]^2\times \left[H^{1}\left(\Omega\right)\cap L^{\infty}\left(\Omega\right)\right]\), such that for any \(\delta^{*}\in\left(0,\delta_{0}\right)\) and initial data \(\left(\mathbf{u}_{0}^{\delta^{*}},\theta_{0}^{\delta^{*}}\right):=\delta^{*}\left(\mathbf{u}_{0},\theta_{0}\right)\), the strong solution \(\left(\mathbf{V}^{\delta^{*}},\rho^{\delta^{*}}\right)\in C\left(0,T_{\text{max}},\left[H^{1}\left(\Omega\right)\right]^2\times L^{2}\left(\Omega\right)\right)\) of the problem \eqref{raodong1219} subject to the initial data \(\left(\mathbf{u}_{0}^{\delta^{*}},\theta_{0}^{\delta^{*}}\right)\) satisfies 
\[
\left\|\rho^{\delta^{*}}\left(T^{\delta^{*}}\right)\right\|_{L^{1}\left(\Omega\right)}\geq \epsilon,~
\left\|\mathbf{V}^{\delta^{*}}\left(T^{\delta^{*}}\right)\right\|_{L^1\left(\Omega\right)}\geq \epsilon,
\]
for some escape time \(0<T^{\delta^{*}}<T_{max}\), where \(T_{max}\) is the maximal existence time of \(\left(\mathbf{V}^{\delta^{*}},
\rho^{\delta^{*}}\right)\).
\end{theorem}
% In the proof of Theorem \ref{hadamardyiyixia0202}, we firstly estimate the velocity, then we estimate density in the proof of Theorem \ref{hadamardyiyixia0202}. By this strategy, the nonlinear instability in Hadamard sense can be verified.
 \subsection{Linear and nonlinear stability}
 We now turn our attention to the stability analysis of the steady-state solutions of the system. This analysis is carried out under two distinct scenarios: linear stability when \(h\left(x,y\right)<0\) throughout the domain \(\Omega\), and nonlinear stability when \(h\left(x,y\right)\) is a negative constant.
\begin{theorem}\label{wendingxing0225}[{\bf Linear and nonlinear stability}] Assume that \(h<0\) throughout the domain  $\Omega$. We have the following conclusions.
\begin{enumerate}
\item The linearized system \eqref{xianxing1219} is global stable. That is, for any initial value \(\left(\rho\left(0\right),\mathbf{V}\left(0\right)\right)\in \left(L^{\infty}\left(\Omega\right)\cap H^{1}\left(\Omega\right)\right)\times
\left(\left[H^2\left(\Omega\right)\right]^2\cap \left[H_{0}^{1}\left(\Omega\right)\right]^2\right)\) and \(\nabla\cdot \mathbf{V}\left(0\right)=0,\) the strong solution \(\left(\rho,\mathbf{V}\right)\) of the linearized system \eqref{xianxing1219} satisfies the following estimates 
\begin{align}\label{xuyaode0226}
\begin{aligned}
\left\|\rho\right\|_{L^2\left(\Omega\right)}^2
+\left\|\mathbf{V}_{t}\right\|_{L^2\left(\Omega\right)}^2
&+\left\|\mathbf{V}\right\|_{H^2\left(\Omega\right)}^2+\int_{0}^{t}
\left(\left\|\nabla\mathbf{V}\right\|_{L^2\left(\Omega\right)}^2+\left\|\nabla\mathbf{V}_{s}\right\|_{L^2\left(\Omega\right)}^2\right)ds
\\
&\leq C\left(\left\|\rho\left(0\right)\right\|_{L^2\left(\Omega\right)}^2+\left\|\mathbf{V}\left(0\right)\right\|_{H^2\left(\Omega\right)}^2\right),
\end{aligned}
\end{align}
where \(C=C\left(\mu,f,\rho_{0},h,\Omega\right)>0\) is independent of \(t\). 
Furthermore, we have
\begin{align}\label{bijin0226}
\left\|\mathbf{V}\right\|_{H^1\left(\Omega\right)}\rightarrow 0,~\text{as~}t\rightarrow+\infty.
\end{align}

\item If we further assume that 
\(h\equiv const<0\), then for any \(\sigma\), \(K\), there is a \(\hat{\delta}:=\hat{\delta}\left(K,\sigma\right)>0\) sufficient small, such that for any initial value \(\left(\rho\left(0\right),\mathbf{V}\left(0\right)\right)\in \left(L^{\infty}\left(\Omega\right)\cap H^{1}\left(\Omega\right)\right)\times
\left[H^2\left(\Omega\right)\right]^2\), \(\mathbf{V}\left(0\right)|_{\partial\Omega}=\mathbf{0}\) and \(\nabla\cdot \mathbf{V}\left(0\right)=0\) satisfying 
\[
\left\|\rho\left(0\right)\right\|_{L^2\left(\Omega\right)}^2+\left\|\mathbf{V}\left(0\right)\right\|_{H^2\left(\Omega\right)}^2
\leq \hat{\delta}^2
\]
and 
\[
0<\sigma\leq\inf\limits_{\left(x,y\right)\in\Omega}\left\{\rho\left(0\right)+\rho_{0}\right\}
\leq
\sup\limits_{\left(x,y\right)\in\Omega}\left\{\rho\left(0\right)+\rho_{0}\right\}\leq K,
\]
there exists a unique global strong solution
\(\left(\rho,\mathbf{V}\right)\) satisfying the following estimate 
\begin{align*}
\begin{aligned}
\left\|\mathbf{V}_{t}\right\|_{L^2\left(\Omega\right)}^2+\left\|\rho\right\|_{L^2\left(\Omega\right)}^2
+\left\|\mathbf{V}\right\|_{H^2\left(\Omega\right)}^2&+\int_{0}^{t}\left[\left\|\nabla\mathbf{V}\right\|^2+\left\|\nabla\mathbf{V}_{s}\right\|^2\right]ds
\\
&\leq C\left(\left\|\rho\left(0\right)\right\|_{L^2\left(\Omega\right)}^2+\left\|\mathbf{V}\left(0\right)\right\|_{H^2\left(\Omega\right)}^2\right),
\end{aligned}
\end{align*}
where \(C=C\left(\mu,\Omega,K,f,h,\hat{\delta},\sigma\right)>0\) is independent of \(t\). 
And we have 
\[
\left\|\mathbf{V}\right\|_{H^{1}\left(\Omega\right)}\rightarrow 0,~\text{as~}t\rightarrow+\infty.
\]
\end{enumerate}
\end{theorem}

\subsection{Organization of the paper}
The remainder of this paper is organized as follows: Section \ref{chubu0202} presents several essential lemmas including the well-posedness and some useful inequalities;  Section \ref{dingli0304} is dedicated to providing the detailed proofs of the main theorems. Specifically, subsection \ref{xianxing0203} analyzes the linear instability to this system \eqref{raodong1219}; The subsections \ref{feixianxing0203} and \ref{nonlinear0203} discuss the nonlinear instability in Lipschitz and Hadamard senses, respectively. 
The subsection \ref{wendingxing0226} provides the proofs of linear and nonlinear stability. Finally,
Section \ref{appendix0304} offers the proofs of the lemmas presented in the Section \ref{chubu0202}.
\section{Preliminaries}\label{chubu0202}
In this section, we present several fundamental lemmas that are essential for the subsequent analysis of the stability and instability of the system \eqref{moxing1121}. The proofs of these lemmas can be found in Appendix \ref{appendix0304}. For those lemmas without proofs, the corresponding references will be indicated.
\subsection{Well-posedness}
We first recall the well-posedness of the system \eqref{moxing1121}. The following lemma, which is adapted from \cite{hkim1987,Choe2003}, gives the existence and uniqueness of solutions under appropriate initial data.
\begin{lemma}\label{shidingxingwenti0202} 
 Let $\Omega\subset \R^2$ be a bounded domain with smooth boundary.
If the initial data \(\left(\mathbf{V},\rho\right)|_{t=0}=\left(\mathbf{V}\left(0\right),\rho\left(0\right)\right)\in \left[H^2\left(\Omega\right)\right]^2\cap \left(H^{1}\left(\Omega\right)\cap L^{\infty}\left(\Omega\right)\right)\), where \(\mathbf{V}\left(0\right)|_{\partial\Omega}=\mathbf{0}\), \(\nabla\cdot\mathbf{V}\left(0\right)=0\) and \(\rho\left(0\right)>\sigma>0\) (\(\sigma\) is an arbitrarily given positive constant), then there exists a positive constant \(T^{*}\) such that 
the system \eqref{moxing1121} has a unique strong solution \(\left(\mathbf{V},\rho\right)\) satisfying 
\begin{align}\label{xiekaidian0202}
    \begin{aligned}
      &\sqrt{\rho}\partial_{t}\mathbf{V}\in L^{\infty}\left(0,T;\left[L^2
      \left(\Omega\right)\right]^2\right);~
      \partial_{t}\mathbf{V}\in L^{2}\left(0,T;\left[H^{1}\left(\Omega\right)\right]^2\right);
      \\
      &\mathbf{V}\in L^{\infty}\left(0,T;\left[H^{2}\left(\Omega\right)\right]^2\right)\cap 
      L^{2}\left(0,T;\left[W^{2,4}\left(\Omega\right)\right]^2\right);
      \\
      &\nabla p\in L^{\infty}\left(0,T;\left[L^{2}\left(\Omega\right)\right]^2\right)
      \cap L^{2}\left(0,T;\left[L^{4}\left(\Omega\right)\right]^2\right),
      \\
      &\rho\in L^{\infty}\left(0,T;H^1\left(\Omega\right)\right),~
    \partial_{t}\rho\in L^{\infty}\left(0,T;L^{2}\left(\Omega\right)\right),
    \end{aligned}
  \end{align}
  where \(0<T<T^{*}\).
\end{lemma}
The proof of this lemma involves advanced techniques from the theory of partial differential equations, such as energy estimates, fixed - point theorems, and the properties of Sobolev spaces. The details can be found in the references\cite{hkim1987,Choe2003}.
% \begin{remark}
%     For the details, one can refer to \cite{hkim1987,Choe2003}.
% \end{remark}
\subsection{Some useful inequalities}
We also need some essential inequalities. These inequalities will be used to estimate the norms of the solutions and their derivatives, which are crucial for the stability and instability results.
\begin{lemma}\label{poincarebudengshi0201}
  Let \(\mathbf{V}\in \left[H_{0}^{1}\left(\Omega\right)\right]^2\). Then for each \(k\in \mathbf{N}^{*}\), we have 
  \begin{align*}
    \begin{aligned}
    \left\|\mathbf{V}\right\|_{L^{2k}\left(\Omega\right)}\leq C 
    \left\|\nabla\mathbf{V}\right\|_{L^2\left(\Omega\right)},
  \end{aligned}
\end{align*}
  where \(C=C\left(k,\Omega\right)\) is a positive constant depending on \(k,\Omega\).
\end{lemma}

The proofs of Lemmas \ref{yuanlaihaishao0201} and \ref{zuidazhi0201} can be found in \cite{li_global_2021}.
\begin{lemma}\label{yuanlaihaishao0201} 
  Let \(\mathbf{V}\in \left[H_{0}^{1}\left(\Omega\right)\right]^2\). Then, there exists a positive constant 
  \(C=C\left(\Omega\right)\) satisfying 
  \[
  \left\|\mathbf{V}\right\|_{L^4\left(\Omega\right)}^2\leq C\left 
  \|\nabla\mathbf{V}\right\|_{L^2\left(\Omega\right)}
  \left\|\mathbf{V}\right\|_{L^2\left(\Omega\right)}.   
  \]
\end{lemma}
\begin{lemma}\label{zuidazhi0201} 
  Let \(\mathbf{V}\in \left[H_{0}^{1}\left(\Omega\right)\right]^2\cap \left[H^2\left(\Omega\right)\right]^2\). Then, there exists a positive constant 
  \(C=C\left(\Omega\right)\) satisfying 
  \[
  \left\|\mathbf{V}\right\|_{L^\infty\left(\Omega\right)}^2\leq C\left\|\mathbf{V}\right\|_{L^2\left(\Omega\right)}\left\|\nabla^2\mathbf{V}\right\|_{L^2\left(\Omega\right)}.   
  \]
\end{lemma}
\begin{lemma}\label{feichangzhongyao0226}
Let \(v\in H^2\left(\Omega\right)\cap H_{0}^{1}\left(\Omega\right)\). Then, we have 
\begin{align*}
\left\|\nabla v\right\|_{L^2\left(\Omega\right)}^2
\leq \left\|v\right\|_{L^2\left(\Omega\right)}
\left\|\nabla^2 v\right\|_{L^2\left(\Omega\right)}.
\end{align*}
\end{lemma}

\begin{lemma}\label{budengshidezhengming0201}
    If the following equality holds, 
    \begin{align*}
      \left\|\nabla\mathbf{V}\right\|_{L^2\left(\Omega\right)}^2 
      \leq C\int_{0}^{t}\left\|\nabla\mathbf{V}\right\|_{L^2\left(\Omega\right)}^6ds
      +C_{1}(t+1),
    \end{align*}
    where \(C\) and \(C_{1}\) are positive constants and \(C_{1}\) is small enough, then there exist a positive constant \(C^*\) and \(T^{*}>0\) such that when 
    \(t\leq T^{*}\),
    \begin{align*}
      \left\|\nabla\mathbf{V}\left(t\right)\right\|_{L^2\left(\Omega\right)}^2
      \leq C^{*}.
    \end{align*}
  \end{lemma}
  
\section{Proofs of main theorems}\label{dingli0304}
This section is dedicated to presenting the detailed proofs of the main theorems, which are crucial for establishing the instability and stability properties of the steady-state solutions of the considered fluid system \eqref{moxing1121}.
\subsection{Proof of the Theorems \ref{xianxingbuwending1221}}
\label{xianxing0203}
In this subsection, we aim to prove that the steady-state solution \(\left(\mathbf{0}, P_{0},\rho_{0}\right)\)
is linearly unstable. To achieve this, we employ the method of separation of variables, as outlined in \eqref{yun1219}, to eliminate the variable \(t\). After performing a series of calculations to remove the unknown
\(\theta\), we analyze the partial differential equation \eqref{wanzhe1219}, which is solved using the modified variational method. Specifically, we verify that a family of functionals derived from \eqref{wanzhe1219} attains an extreme value, as shown in Proposition \ref{pingfan1220}, and the corresponding extreme point represents the solution to the system \eqref{wanzhe1219}, as indicated in Proposition \ref{jiujiu0204}. The existence of the growth rate \(\Lambda\)
is established through the intermediate value theorem for continuous functions, as demonstrated in Proposition \eqref{xiangdeng1221}. Therefore, we will examine the properties of
\(\Phi\left(s\right)\), as discussed in Propositions \eqref{dayuling1220} to \eqref{zizhuo1221}.

Subsequently, let 
\begin{align}\label{yun1219}
\rho=e^{\lambda t}\theta\left(x,y\right),~
\mathbf{V}=e^{\lambda t}\mathbf{u}\left(x,y\right),~
P=e^{\lambda t}\pi\left(x,y\right),
\end{align}
and substitute \eqref{yun1219} into \eqref{xianxing1219}, We obtain 
\begin{align}\label{jixu1219}
\begin{cases}
\lambda\theta+h\mathbf{u}\cdot\nabla f=0,
\\
\lambda\rho_{0}\mathbf{u}+\nabla\pi=\mu\Delta\mathbf{u}-\theta\nabla f,
\\
\nabla\cdot\mathbf{u}=0,~\mathbf{u}|_{\partial\Omega}=\mathbf{0}.
\end{cases}
\end{align}
By multiplying the first equation of \eqref{jixu1219} by $\lambda$ and eliminating $\theta$, we have 
\begin{align}\label{wanzhe1219}
\begin{cases}
\lambda^2\rho_{0}\mathbf{u}+\lambda\nabla\pi=\lambda\mu\Delta\mathbf{u}+h\left(\mathbf{u}\cdot\nabla f\right)\nabla f,
\\
\nabla\cdot\mathbf{u}=0,~\mathbf{u}|_{\partial\Omega}=0.
\end{cases}
\end{align}
We consider the following extreme value problem, 
\begin{align}\label{zhigeye1219}
\Phi\left(s\right)=\sup\limits_{\mathbf{u}\in \mathbf{H}_{div}}\frac{E\left(\mathbf{u},s\right)}{J\left(\mathbf{u}\right)}=\sup\limits_{\mathbf{u}\in\mathcal{A}}E\left(\mathbf{u},s\right),
\end{align}
where 
\begin{align*}
\begin{aligned}
&E\left(\mathbf{u},s\right):=E\left(\mathbf{u}\right)=-s\mu E_{1}\left(\mathbf{u}\right)+E_{2}\left(\mathbf{u}\right),~s>0,
\\
&E_{1}\left(\mathbf{u}\right)=\left\|\nabla\mathbf{u}\right\|_{L^2\left(\Omega\right)}^2,\quad E_{2}\left(\mathbf{u}\right)=\int_{\Omega}h\left|\mathbf{u}\cdot\nabla f\right|^2dxdy,~J\left(\mathbf{u}\right)
=\int_{\Omega}\rho_{0}\mathbf{u}^2 dxdy,
\\
&\mathbf{H}_{\text{div}}=\left\{\mathbf{u}\in \left[H_{0}^1\left(\Omega\right)\right]^2|\nabla\cdot\mathbf{u}=0\right\},~\mathcal{A}=\left\{\mathbf{u}\in \mathbf{H}_{\text{div}}|J\left(\mathbf{u}\right)=1\right\}.
\end{aligned}
\end{align*}

\begin{proposition}\label{pingfan1220} For any \(s>0\),
\(E\left(\mathbf{u}\right)\) achieves the maximum on \(\mathcal{A}\).
\end{proposition}
\begin{proof}
For any \(\mathbf{u}\in\mathcal{A}\), we have 
\begin{align}\label{zuidazhi1220}
E\left(\mathbf{u}\right)
\leq \int_{\Omega}h\left|\mathbf{u}\cdot\nabla f\right|^2dxdy=\int_{\Omega}\frac{h}{\rho_{0}}\rho_{0}\left|\mathbf{u}\cdot \nabla f\right|^2 dxdy\leq \left\|\frac{h}{\rho_{0}}\left|\nabla f\right|^2\right\|_{L^{\infty}\left(\Omega\right)},
\end{align}
which indicates the existence of an upper-bound of \(E\left(\mathbf{u}\right)\) on \(\mathcal{A}\). Thus, there exists a sequence \(\left\{\mathbf{u}_{n}\right\}\subset \mathcal{A}\) such that 
\[
\lim\limits_{n\rightarrow+\infty}
E\left(\mathbf{u}_{n}\right)=\sup\limits_{\mathbf{u}\in\mathcal{A}}E\left(\mathbf{u}\right).
\]
Consequently, \(\left\{\mathbf{u}_{n}\right\}\) is bounded in \(\mathbf{H}_{\text{div}}\). By the reflexivity, there exists a \(\mathbf{u}_{0}\in \mathbf{H}_{\text{div}}\) such that 
\[
\mathbf{u}_{n}\rightarrow \mathbf{u}_{0}~\text{weakly~in~}\mathbf{H}_{\text{div}}~\text{and~strongly~in}~\left[L^{2}\left(\Omega\right)\right]^2.
\]
Since it is easy to verify that \(E\left(\mathbf{u}\right)\) is weak upper semi-continuous by convexity and strong convergence, we have
\begin{align*}
E\left(\mathbf{u}_{0}\right)\geq\lim\limits_{n\rightarrow+\infty}E\left(\mathbf{u}_{n}\right)=\sup\limits_{\mathbf{u}\in\mathcal{A}}E\left(u\right),
\end{align*}
and \(\mathbf{u}_{0}\in\mathcal{A}\). Therefore,  
\[
E\left(\mathbf{u}_{0}\right)=\sup\limits_{\mathbf{u}\in\mathcal{A}}E\left(\mathbf{u}\right).
\]
\end{proof}

\begin{proposition}\label{jiujiu0204} If for some \(\Lambda>0\), \(\Lambda^2 J\left(\mathbf{u}_{0}\right)=\sup\limits_{\mathbf{u}\in\mathcal{A}}E\left(\mathbf{u},\Lambda\right)=E\left(\mathbf{u}_{0},\Lambda\right)\), then \(\mathbf{u}_{0}\) solves the equation \eqref{wanzhe1219} as \(\lambda=\Lambda\) and \(\mathbf{u}_{0}\) is smooth.
\end{proposition}
\begin{proof}
\(\forall~\mathbf{w}\in \mathbf{H}_{\text{div}}\), \(\eta\) and \(\widetilde{\sigma}\in\mathbf{R}\), define 
\[
F\left(\eta,\widetilde{\sigma}\right)
=J\left(\mathbf{u}_{0}+\eta\mathbf{w}/\sqrt{J\left(\mathbf{w}\right)}+\widetilde{\sigma} \mathbf{u}_{0}\right).
\]
Then, 
\[
F\left(0,0\right)=1,~\partial_{\eta}F\left(0,0\right)=2\int_{\Omega}\rho_{0}\mathbf{u}_{0}\cdot\mathbf{w}/\sqrt{J\left(\mathbf{w}\right)}dxdy,~\partial_{\widetilde{\sigma}}F\left(0,0\right)=2\neq 0.
\]
As a result, by the implicit function theorem, there exists a smooth curve \(\widetilde{\sigma}=\widetilde{\sigma}\left(\eta\right)\) defined in the neighborhood of \(\eta=0\) such that \begin{align}\label{toutou1221}
0=\widetilde{\sigma}\left(0\right),~
\widetilde{\sigma}'\left(0\right)=-\int_{\Omega}\rho_{0}\mathbf{u}_{0}\cdot\mathbf{w}/\sqrt{J\left(\mathbf{w}\right)}dxdy.
\end{align}
Subsequently, we define 

\[
\widetilde{F}\left(\eta\right) 
=E\left(\mathbf{u}_{0}+\eta \mathbf{w}/\sqrt{J\left(\mathbf{w}\right)}+\widetilde{\sigma}\left(\eta\right)\mathbf{u}_{0},\Lambda\right).
\]
Since \(\mathbf{u}_{0}\) is the extreme value point, we have 
\begin{align*}
\begin{aligned}
0=\frac{d\widetilde{F}}{d\eta}|_{\eta=0}=&-2\Lambda\mu\int_{\Omega}\nabla\mathbf{u}_{0}\cdot\nabla \left(\mathbf{w}/J\left(\mathbf{w}\right)\right)dxdy+2\int_{\Omega}h\mathbf{u}_{0}\cdot\nabla f\left(\mathbf{w}/\sqrt{J\left(\mathbf{w}\right)}\cdot\nabla f\right)dxdy
\\
&-2\widetilde{\sigma}'\left(0\right)\left[\Lambda \mu \int_{\Omega}\left|\nabla\mathbf{u}_{0}\right|^2dxdy-\int_{\Omega}h\left|\mathbf{u}_{0}\cdot\nabla f\right|^2dxdy\right],
\end{aligned}
\end{align*}
implying that 
\[
-2\Lambda\mu\int_{\Omega}\nabla\mathbf{u}_{0}\cdot\nabla \mathbf{w}dxdy+2\int_{\Omega}h\mathbf{u}_{0}\cdot\nabla f\left(\mathbf{w}\cdot\nabla f\right)dxdy=\Lambda^2\int_{\Omega}\rho_{0}\mathbf{u}_{0}\cdot\mathbf{w}dxdy.
\]
This verifies that \(\mathbf{u}_{0}\) is the weak solution to the system \eqref{wanzhe1219} as \(\lambda=\Lambda\). Furthermore, since \(\partial\Omega\), \(f\) and \(h\) are smooth, then by the standard bootstrap method, we have \(\mathbf{u}_{0}\in \left[H^{k}\left(\Omega\right)\right]^2\) for any \(k\in\mathbf{N}^{*}\). Thus, \(\mathbf{u}_{0}\) is smooth.
\end{proof}
\begin{remark}\label{buweiling0202}
Since \(\mathbf{u}_{0}=\left(u_{01},u_{02}\right)\in \mathbf{H}_{\text{div}}\) and \(\rho_{0}>0\), then
\(\left\|u_{01}\right\|_{L^{2}\left(\Omega\right)}\left\|u_{02}\right\|_{L^2\left(\Omega\right)}\neq 0\). Indeed, assume that \(u_{01}\equiv 0\),  then \(\int_{\Omega}\rho_{0}\mathbf{u}_{0}^2dxdy>0\) means \(u_{02}\neq 0\). However, from 
the condition 
\(\partial_{x}u_{01}+\partial_{y}u_{02}=0\) and \(\mathbf{u}_{0}|_{\partial\Omega}=\mathbf{0}\), we can deduce that \(u_{02}\equiv 0\), which is a contradiction. And let \(\theta_{0}=-\frac{\mathbf{u}_{0}\cdot\nabla\rho_{0}}{\Lambda}\), then \(\left\|\theta_{0}\right\|_{L^2\left(\Omega\right)}\neq 0\) and \(\theta_{0}\) is smooth. Then,
\(\left(\mathbf{u}_{0},\theta_{0}\right)\) solves equations \eqref{jixu1219} at \(\lambda=\Lambda\).
\end{remark}

\begin{proposition}\label{dayuling1220}
    Under the conditions \eqref{xuyao0207} and \(h\left(x_{0},y_{0}\right)>0\), there exists a positive constant \(M>0\) such that \(\left|\Phi\left(s\right)\right|\leq M\) for \(s>0\). Moreover, for any sufficiently small \(s>0\), 
    \(0<\Phi\left(s\right)\leq M\).
\end{proposition}
\begin{proof}
First, according to Proposition \ref{pingfan1220}, \(\Phi\left(s\right)\) is well-defined. Then for any \(s>0\), we assume that 
\(\Phi\left(s\right)\) is unbounded. From \eqref{zuidazhi1220}, we have 
    \begin{align*}
    E\left(\mathbf{u}\right)\leq \left\|\frac{h}{\rho_{0}}\left|\nabla f\right|^2\right\|_{L^{\infty}\left(\Omega\right)}.
    \end{align*}
    Thus, there exists a \(s>0\) such that \(\Phi\left(s\right)=-\infty\) implying that \(E\left(\mathbf{u}\right)=-\infty\) for any \(\mathbf{u}\in\mathcal{A}\), which is impossible obviously. It follows that there exists a positive constant \(M>0\) such that 
    \begin{align}\label{houmian1221}
    \left|\Phi\left(s\right)\right|\leq M.
    \end{align}
    
    From the condition \(h\left(x_{0},y_{0}\right)>0\), there exists a \(\sigma>0\) such that 
    \(\forall~\left(x,y\right)\in B_{\sigma}\left(\left(x_{0},y_{0}\right)\right)\), we have 
    \[
    h\left(x,y\right)>0.
    \]
Thus, for any \(s>0\) small enough, we can choose \(\mathbf{u}\in C_{0}^{\infty}\left(B_{\sigma}\left(\left(x_{0},y_{0}\right)\right)\right)\) such that 
\[E\left(\mathbf{u}\right)>0.\]
Consequently, \(\Phi\left(s\right)>0\). This completes the proof.
\end{proof}

\begin{proposition}\label{zizhuo1221}
    \(\Phi\left(s\right)\) is decreasing in \(\left(0,+\infty\right)\) and \(\Phi\left(s\right)\in C_{\text{loc}}^{0,1}\left(0,+\infty\right)\).
\end{proposition}
\begin{proof}
\(\forall~0\leq s_{1}\leq s_{2}\), there exist \(\mathbf{u}_{1}\) and \(\mathbf{u}_{2}\) in \(\mathbf{H}_{\text{div}}\) such that 
\begin{align}\label{xinwuwaiwu1221}
\begin{aligned}
\Phi\left(s_{1}\right)-&\Phi\left(s_{2}\right)
=\frac{E\left(\mathbf{u}_{1},s_{1}\right)}{J\left(\mathbf{u}_{1}\right)}-\frac{E\left(\mathbf{u}_{2},s_{2}\right)}{J\left(\mathbf{u}_{2}\right)}
\\
&\geq \frac{E\left(\mathbf{u}_{2},s_{1}\right)-E\left(\mathbf{u}_{2},s_{2}\right)}{J\left(\mathbf{u}_{2}\right)}
=\frac{\left(s_{2}-s_{1}\right)\mu\int_{\Omega}\left|\nabla\mathbf{u}_{2}\right|^2dxdy}{J\left(\mathbf{u}_{2}\right)}\geq 0.
\end{aligned}
\end{align}
Thus, \(\Phi\left(s\right)\) is decreasing in \(\left(0,+\infty\right)\). 

From the Proposition \ref{dayuling1220}, for any \(s>0\), we can conclude that 
\[
\frac{s\mu E_{1}\left(\mathbf{u}\right) }{J\left(\mathbf{u}\right)}\leq M+\left\|\frac{h}{\rho_{0}}\left|\nabla f\right|^2\right\|_{L^{\infty}\left(\Omega\right)}:=\widetilde{M}.
\]
Consequently, for any \(s_{1}\leq s_{2}\) and \(s_{1},s_{2}\in [a,b]\), where \(0<a<b\) are fixed constants, we can deduce that 
by \eqref{xinwuwaiwu1221} 
\begin{align}\label{jiehe1221}
0\leq \Phi\left(s_{1}\right)-\Phi\left(s_{2}\right)\leq  \frac{\left(s_{2}-s_{1}\right)\mu\int_{\Omega}\left|\nabla\mathbf{u}_{1}\right|^2dxdy}{J\left(\mathbf{u}_{1}\right)}\leq 
\frac{\widetilde{M}\left(s_{2}-s_{1}\right)}{a},
\end{align}
which shows that \(\Phi\left(s\right)\in C_{\text{loc}}^{0,1}\left(0,+\infty\right)\). The proof is completed.
\end{proof}

\begin{proposition}\label{xiangdeng1221}
    There exists a unique \(\Lambda>0\) such that \(\Lambda^2=\Phi\left(\Lambda\right)\).
\end{proposition}
\begin{proof}
Define \(H\left(\lambda\right)=\lambda^2-\Phi\left(\lambda\right)\). From the Propositions \ref{dayuling1220} and \ref{zizhuo1221}, we can show that \(H\left(\lambda\right)\) is increasing on \((0,+\infty)\), \(\lim\limits_{\lambda\rightarrow 0^{+}}H\left(\lambda\right)<0\) and \(\lim\limits_{\lambda\rightarrow+\infty}H\left(\lambda\right)>0\). Thus, by the intermediate value theorem for continuous function, we can conclude that there exists a unique \(\Lambda>0\) such that 
\(\Lambda^2=\Phi\left(\Lambda\right)\). The proof is finished.
\end{proof}
\begin{remark}\label{xuyaode0210}
From the definitions of \(\Phi\left(s\right)\) and \(\Lambda\), one can easily obtain 
\begin{align*}
\begin{aligned}
\Lambda^2 \left\|\sqrt{\rho_{0}}\mathbf{u}\right\|_{L^2\left(\Omega\right)}^2 
\geq -\Lambda\mu\left\|\nabla\mathbf{u}\right\|_{L^2\left(\Omega\right)}^2+\int_{\Omega}h\left|\mathbf{u}\cdot\nabla f\right|^2dxdy,~\text{for~any~}\mathbf{u}\in \mathbf{H}_{\text{div}}.
\end{aligned}
\end{align*}
And from Remark \ref{buweiling0202}, we also have 
\begin{align*}
\Lambda^2 \left\|\sqrt{\rho_{0}}\mathbf{u}_{0}\right\|_{L^2\left(\Omega\right)}^2 
= -\Lambda\mu\left\|\nabla\mathbf{u}_{0}\right\|_{L^2\left(\Omega\right)}^2+\int_{\Omega}h\left|\mathbf{u}_{0}\cdot\nabla f\right|^2dxdy.
\end{align*}
\end{remark}

% \begin{align*}
% \begin{aligned}
% \mathcal{Q}_{T}=\bigg{\{}
% \left(\mathbf{V},P,\rho\right)|
% &\mathbf{V}\in L^{\infty}\left(0,T;\left[H^2\left(\Omega\right)\right]^2\right),~\mathbf{V}_{t}\in L^2\left(0,T;\left[H^1\left(\Omega\right)\right]^2\right),
% \\
% &\nabla P\in L^{\infty}\left(0,T;\left[L^2\left(\Omega\right)\right]^2\right),\rho\in L^\infty\left(0,T;H^{1}\left(\Omega\right)\right),~\rho_{t}\in L^{2}\left(0,T;L^{2}\left(\Omega\right)\right)
% \bigg{\}}.
% \end{aligned}
% \end{align*}

\subsection{Proof of the Theorem \ref{lipuxizhi0202}}\label{feixianxing0203}
In this part, we prove Theorem  \ref{lipuxizhi0202}. As introduced in the Introduction, the proof will be divided into several parts. First, we consider the uniqueness of the solution to the linear system \eqref{xianxing1219} and obtain the following conclusion.
\begin{lemma}\label{weiyixing1230}
The linearized system \eqref{xianxing1219} has a unique weak solution when the initial value \(\left(\mathbf{u}_{0},\theta_{0}\right)\in \left(\left[H^2\left(\Omega\right)\right]^2\cap \mathbf{H}_{div}\right)\times \left(H^{1}\left(\Omega\right)\cap L^{\infty}\left(\Omega\right)\right)\).
\end{lemma}
\begin{proof}
The uniqueness of the solution to the system \eqref{xianxing1219} can be derived from the fact that if the initial value \(\left(\mathbf{u}_{0},\theta_{0}\right)=\left(\mathbf{0},0\right)\), the system \eqref{xianxing1219} has only trivial solution. 

Multiplying \(\eqref{xianxing1219}_{1}\) and \(\eqref{xianxing1219}_{2}\) by \(\rho\) and \(\mathbf{V}\) in $L^2$, respectively,  integrating over \(\Omega\) by parts and adding up the results give that 
\begin{align*}
\begin{aligned}
\frac{d}{dt}\int_{\Omega}\left(\rho^2+\rho_{0}\mathbf{V}^2\right)dxdy 
&=-2\int_{\Omega}\mathbf{V}\cdot\nabla\rho_{0}\rho dxdy-2\mu\left\|\nabla\mathbf{V}\right\|_{L^2\left(\Omega\right)}^2-2\int_{\Omega}\rho\nabla f\cdot\mathbf{V}dxdy
\\
&
\leq C\int_{\Omega}\left(\rho^2+\rho_{0}\mathbf{V}^2\right)dxdy,
\end{aligned}
\end{align*}
where \(C=C\left(\rho_{0},\nabla f\right)>0\) and the Cauchy inequality is used. Then we have 
\[
\frac{d}{dt}\left[e^{-Ct}\int_{\Omega}\left(\rho^2+\rho_{0}\mathbf{V}^2\right)dxdy\right]
\leq 0,
\]
which implies that 
\begin{align*}
\mathbf{V}\equiv\mathbf{0},~\rho\equiv 0,~\text{a.e.~in}~\Omega~\text{and}~\text{for~any~t}\in \left[0,T\right].
\end{align*}
\end{proof}
The following is the nonlinear energy estimate of the strong solution to the nonlinear system \eqref{raodong1219} under small initial values.
\begin{proposition}\label{nengliangguji0102}
    There exists a constant \(\delta\in (0,1)\) such that if \(\left\|\theta_{0}\right\|_{H^1\left(\Omega\right)}^2+\left\|\mathbf{u}_{0}\right\|_{H^2\left(\Omega\right)}^2=\delta_{0}^2<\delta^2\), then any strong solution \(\left(\mathbf{V},P,\rho\right)\) to this system \eqref{raodong1219} with initial value 
    \(\left(\mathbf{u}_{0},\theta_{0}\right)\) satisfies 
    \begin{align}\label{guji0102}
    \begin{aligned}
    \sup\limits_{t\in[0,T]} 
    \bigg{\{}
\left\|\mathbf{V}\right\|_{H^2\left(\Omega\right)}^2
&+\left\|\rho\right\|_{H^1\left(\Omega\right)}^2    +\left\|\nabla P\right\|_{L^2\left(\Omega\right)}^2+
\left\|\left(\mathbf{V}_{t},\rho_{t}\right)\right\|_{L^2\left(\Omega\right)}^2
\\
&+\int_{0}^{t}\left[\left\|\mathbf{V}_{s}\right\|_{H^1\left(\Omega\right)}^2+\left\|\nabla\mathbf{V}\right\|_{H^1\left(\Omega\right)}^2\right]ds
    \bigg{\}}\leq C\delta_{0}^2,
    \end{aligned}
    \end{align}
    where \(\left\|\left(\mathbf{V}_{t},\rho_{t}\right)\right\|_{L^2\left(\Omega\right)}^2:=\left\|\mathbf{V}_{t}\right\|_{L^{2}\left(\Omega\right)}^2+\left\|\rho_{t}\right\|_{L^{2}\left(\Omega\right)}^2\),
    \(T<T_{\text{max}}\) and \(T_{\text{max}}\) is the maximum existence time.
\end{proposition}
\begin{proof}
In what follows, we denote by \(C\) a generic positive constant which may depend on \(T\), \(\mu\), \(\theta_{0}\), \(\rho_{0}\), \(\Omega\) and \(f\). 

1) Estimation of \(\left\|\rho\right\|_{L^{q}\left(\Omega\right)}\)(\(1<q\leq\infty\)).

Multiplying \(\eqref{raodong1219}_{1}\) by \(\left|\rho+\rho_{0}\right|^{q-2}\left(\rho+\rho_{0}\right)\) (at present \(q\in (1,+\infty)\)) and then integrating this result over \(\Omega\) (by parts) give that
\begin{align*}
\frac{d}{dt}\int_{\Omega}\left|\rho+\rho_{0}\right|^q dxdy=0,
\end{align*}
which implies that 
\(\left\|\rho+\rho_{0}\right\|_{L^{q}\left(\Omega\right)}=\left\|\theta_{0}+\rho_{0}\right\|_{L^{q}\left(\Omega\right)}\). Therefore, we have 
\begin{align}\label{dengjieguo0102}
\begin{aligned}
\left\|\rho\right\|_{L^{q}\left(\Omega\right)}
\leq \left\|\theta_{0}+\rho_{0}\right\|_{L^{q}\left(\Omega\right)}+\left\|\rho_{0}\right\|_{L^{q}\left(\Omega\right)}.
\end{aligned}
\end{align}
Furthermore, since \(\left\|\theta_{0}+\rho_{0}\right\|_{L^{\infty}\left(\Omega\right)}<\infty\) and \(\left\|\rho_{0}\right\|_{L^{\infty}\left(\Omega\right)}<\infty\), we can deduce that
\begin{align*}
\begin{aligned}
\left\|\rho\right\|_{L^{\infty}\left(\Omega\right)}
\leq \left\|\theta_{0}+\rho_{0}\right\|_{L^{\infty}\left(\Omega\right)}+\left\|\rho_{0}\right\|_{L^{\infty}\left(\Omega\right)}.
\end{aligned}
\end{align*}

2) Estimation of \(\left\|\nabla\mathbf{V}\right\|_{L^2\left(\Omega\right)}\). 

After multiplying \(\eqref{raodong1219}_{2}\) by \(\mathbf{V}_{t}\) in $L^2$ and then integrating by parts, we obtain 
\begin{align}\label{xinzi0102}
\begin{aligned}
\int_{\Omega}\left(\rho+\rho_{0}\right)|\mathbf{V}_{t}|^2 dxdy+\frac{\mu}{2}\frac{d}{dt}\left\|\nabla\mathbf{V}\right\|_{L^{2}\left(\Omega\right)}^2=K_{1}+K_{2},
\end{aligned}
\end{align}
where 
\begin{align*}
\begin{aligned}
K_{1}:=-\int_{\Omega}\left(\rho+\rho_{0}\right)\mathbf{V}\cdot\nabla\mathbf{V}\cdot\mathbf{V}_{t}dxdy,~
K_{2}:=-\int_{\Omega}\rho\nabla f\cdot\mathbf{V}_{t}dxdy.
\end{aligned}
\end{align*}
By utilizing Cauchy inequality for \(K_{1}\) and \(K_{2}\), from \eqref{xinzi0102}  
we can obtain
\begin{align}\label{meiyuan0102}
\begin{aligned}
\int_{\Omega}\left(\rho+\rho_{0}\right)\left|\mathbf{V}_{t}\right|^2dxdy+&\frac{d}{dt}\left\|\nabla\mathbf{V}\right\|_{L^{2}\left(\Omega\right)}^2\leq C\left[
\left\|\mathbf{V}\right\|_{L^{\infty}\left(\Omega\right)}^2\left\|\nabla\mathbf{V}\right\|_{L^{2}\left(\Omega\right)}^2+\left\|\rho\right\|_{L^{2}\left(\Omega\right)}^2
\right]
\\
&\leq C\left[\left\|\mathbf{V}\right\|_{L^2\left(\Omega\right)}^2\left\|\nabla\mathbf{V}\right\|_{L^2\left(\Omega\right)}^4+\varepsilon\left\|\nabla^2\mathbf{V}\right\|_{L^2\left(\Omega\right)}^2+\left\|\rho\right\|_{L^2\left(\Omega\right)}^2\right]
\end{aligned}
\end{align}
where Lemma \ref{zuidazhi0201} and the Cauchy inequality are used.

In addition, we have 
\[
-\mu\Delta\mathbf{V}+\nabla P=-\left(\rho+\rho_{0}\right)\mathbf{V}_{t}-\left(\rho+\rho_{0}\right)\mathbf{V}\cdot\nabla\mathbf{V}-\rho\nabla f,
\]
then from Stokes' estimate and Cauchy inequality, we have 
\begin{align}\label{buzhuoji0102}
\begin{aligned}
\left\|\nabla^2\mathbf{V}\right\|_{L^2\left(\Omega\right)}^2+\left\|\nabla P\right\|_{L^2\left(\Omega\right)}^2
\leq C\left[\left\|\sqrt{\rho+\rho_{0}}\mathbf{V}_{t}\right\|_{L^2\left(\Omega\right)}^2+\left\|\mathbf{V}\right\|_{L^2\left(\Omega\right)}^2\left\|\nabla\mathbf{V}\right\|_{L^2\left(\Omega\right)}^4
+\left\|\rho\right\|_{L^2\left(\Omega\right)}^2\right].
\end{aligned}
\end{align}
Then, multiplying \eqref{meiyuan0102} and \eqref{buzhuoji0102} by two appropriate constants, respectively, and adding up the results give that 
\begin{align}\label{xinzuo0102}
\begin{aligned}
\int_{\Omega}\left(\rho+\rho_{0}\right)\left|\mathbf{V}_{t}\right|^2dx dy+\frac{d}{dt}\left\|\nabla\mathbf{V}\right\|_{L^2\left(\Omega\right)}^2&\leq C\left[\left\|\mathbf{V}\right\|_{L^2\left(\Omega\right)}^2\left\|\nabla\mathbf{V}\right\|_{L^2\left(\Omega\right)}^4+\left\|\rho\right\|_{L^2\left(\Omega\right)}^2\right]
\\
&\leq C\left\|\nabla\mathbf{V}\right\|_{L^2\left(\Omega\right)}^6+C,
\end{aligned}
\end{align}
where \eqref{dengjieguo0102} and Lemma \ref{poincarebudengshi0201} are used. As a result, we have 
\begin{align}\label{from0102}
\left\|\nabla\mathbf{V}\right\|_{L^2\left(\Omega\right)}^2\leq C\int_{0}^{t}\left\|\nabla\mathbf{V}\right\|_{L^2\left(\Omega\right)}^{6}ds+Ct+C\delta_{0}^2,
\end{align}
which together with the Lemma \ref{budengshidezhengming0201} yields that 
there exists \(T^{*}>0\) such that when \(T<T^{*}\), 
\begin{align}\label{pengyou0102}
\begin{aligned}
\left\|\nabla\mathbf{V}\left(t\right)\right\|_{L^2\left(\Omega\right)}^2<C,~\text{for~any~}t<T.
\end{aligned}
\end{align}
Thus, from \eqref{xinzuo0102}, we have 
\begin{align}\label{buyao0102}
\begin{aligned}
\frac{d}{dt}\left\|\nabla\mathbf{V}\right\|_{L^2\left(\Omega\right)}^2\leq C\left[\left\|\nabla\mathbf{V}\right\|_{L^2\left(\Omega\right)}^2+\left\|\rho\right\|_{L^2\left(\Omega\right)}^2\right].
\end{aligned}
\end{align}
In addition, multiplying \(\eqref{raodong1219}_{1}\) by \(\rho\), integrating over \(\Omega\) by parts and from Cauchy inequality give that 
\begin{align*}
\frac{d}{dt}\left\|\rho\right\|_{L^2\left(\Omega\right)}^2
\leq C\left[\left\|\nabla\mathbf{V}\right\|_{L^2\left(\Omega\right)}^2+\left\|\rho\right\|_{L^2\left(\Omega\right)}^2\right],
\end{align*}
which together with \eqref{buyao0102} yields that 
\begin{align*}
\begin{aligned}
\frac{d}{dt}\left[\left\|\nabla\mathbf{V}\right\|_{L^2\left(\Omega\right)}^2+\left\|\rho\right\|_{L^2\left(\Omega\right)}^2\right]\leq C\left[\left\|\nabla\mathbf{V}\right\|_{L^2\left(\Omega\right)}^2+\left\|\rho\right\|_{L^2\left(\Omega\right)}^2\right].
\end{aligned}
\end{align*}
Thus, from Gronwall inequality, we have 
\begin{align}\label{zaijian0102}
\begin{aligned}
\left\|\nabla\mathbf{V}\right\|_{L^2\left(\Omega\right)}^2+\left\|\rho\right\|_{L^2\left(\Omega\right)}^2\leq C\delta_{0}^2.
\end{aligned}
\end{align}

3) Estimations of \(\left\|\mathbf{V}_{t}\right\|_{L^2\left(\Omega\right)}\), \(\left\|\nabla^2\mathbf{V}\right\|_{L^2\left(\Omega\right)}\) and \(\left\|\nabla P\right\|_{L^2\left(\Omega\right)}\).

From \(\eqref{raodong1219}_{2}\), we obtain 
\begin{align*}
\begin{aligned}
\int_{\Omega}\left(\rho+\rho_{0}\right)\mathbf{V}_{t}^2dxdy=\mu\int_{\Omega}\Delta\mathbf{V}\cdot\mathbf{V}_{t}dxdy-\int_{\Omega}\rho\nabla f\cdot\mathbf{V}_{t} dxdy-\int_{\Omega}\left(\rho+\rho_{0}\right)\mathbf{V}\cdot\nabla\mathbf{V}\cdot\mathbf{V}_{t}dxdy,
\end{aligned}
\end{align*}
which together with Cauchy inequality and 
the assumption that \(\theta_{0}+\rho_{0}>\sigma\), where \(\sigma>0\) is a constant, yields that  
\begin{align}\label{yuanliangwo0102}
\lim\limits_{t\rightarrow 0^{+}}\int_{\Omega}\left(\rho+\rho_{0}\right)\mathbf{V}_{t}^2dxdy\leq C\delta_{0}^2.
\end{align}
We differentiate \(\eqref{raodong1219}_{2}\) with respect to \(t\), multiply this result by \(\mathbf{V}_{t}\), then integrate over \(\Omega\) by parts and obtain 
\begin{align}\label{zaodianxiuxi0102}
\frac{d}{dt}\int_{\Omega}\left(\rho+\rho_{0}\right)\mathbf{V}_{t}^2dxdy+2\mu\left\|\nabla\mathbf{V}_{t}\right\|_{L^2\left(\Omega\right)}^2=\sum\limits_{i=3}^{7}K_{i},
\end{align}
where 
\begin{align*}
\begin{aligned}
&K_{3}=2\int_{\Omega}\mathbf{V}\cdot\nabla\left(\rho+\rho_{0}\right)\nabla f\cdot \mathbf{V}_{t}dxdy=-2\int_{\Omega}\left(\rho+\rho_{0}\right)\mathbf{V}\cdot\nabla\left(\nabla f\cdot\mathbf{V}_{t}\right)dxdy,
\\
&K_{4}=2\int_{\Omega}\mathbf{V}\cdot\nabla\left(\rho+\rho_{0}\right)\mathbf{V}\cdot\nabla\mathbf{V}\cdot\mathbf{V}_{t}
dxdy=-2\int_{\Omega}
\left(\rho+\rho_{0}\right)\mathbf{V}\cdot\nabla
\left[\left(\mathbf{V}\cdot\nabla\mathbf{V}\cdot\mathbf{V}_{t}\right)\right]
dxdy,
\\
&K_{5}=-2\int_{\Omega}\left(\rho+\rho_{0}\right)\mathbf{V}_{t}\cdot\nabla\mathbf{V}\cdot\mathbf{V}_{t}dxdy,~K_{6}=-2\int_{\Omega}\left(\rho+\rho_{0}\right)\mathbf{V}\cdot\nabla\mathbf{V}_{t}\cdot\mathbf{V}_{t}dxdy,
\\
&K_{7}=\int_{\Omega}\mathbf{V}\cdot\nabla\left(\rho+\rho_{0}\right)\mathbf{V}_{t}^2 dxdy
=-\int_{\Omega}\left(\rho+\rho_{0}\right)\mathbf{V}\cdot\nabla
\left(\mathbf{V}_{t}^2\right)dxdy.
\end{aligned}
\end{align*}
From Cauchy inequality and Lemmas \ref{poincarebudengshi0201}-\ref{zuidazhi0201}, we obtain 
\begin{align*}
\begin{aligned}
&K_{3}\leq C\left\|\mathbf{V}\right\|_{L^{2}\left(\Omega\right)}^2
+\varepsilon\left\|\nabla\mathbf{V}_{t}\right\|_{L^2\left(\Omega\right)}^2,~K_{4}\leq C\left\|\nabla\mathbf{V}\right\|_{L^2\left(\Omega\right)}^4\left\|\nabla^2\mathbf{V}\right\|_{L^2\left(\Omega\right)}^2
+\varepsilon\left\|\nabla\mathbf{V}_{t}\right\|_{L^2\left(\Omega\right)}^2,
\\
&K_{5},K_{6},K_{7}\leq C\left\|\nabla\mathbf{V}\right\|_{L^2\left(\Omega\right)}^2\left\|\sqrt{\rho+\rho_{0}}\mathbf{V}_{t}\right\|_{L^2\left(\Omega\right)}^2
+\varepsilon\left\|\nabla\mathbf{V}_{t}\right\|_{L^2\left(\Omega\right)}^2.
\end{aligned}
\end{align*}
Thus, from \eqref{zaodianxiuxi0102} and \eqref{zaijian0102}, we have 
\begin{align*}
\begin{aligned}
\frac{d}{dt}\left\|\sqrt{\rho+\rho_{0}}\mathbf{V}_{t}\right\|_{L^2\left(\Omega\right)}^2+\left\|\nabla\mathbf{V}_{t}\right\|_{L^2\left(\Omega\right)}^2
\leq C\left[
\left\|\mathbf{V}\right\|_{L^2\left(\Omega\right)}^2
+\left\|\nabla^2\mathbf{V}\right\|_{L^2\left(\Omega\right)}^2
+\left\|\sqrt{\rho+\rho_{0}}\mathbf{V}_{t}\right\|_{L^2\left(\Omega\right)}^2
\right],
\end{aligned}
\end{align*}
which together with \eqref{buzhuoji0102} and \eqref{zaijian0102} yields that
\begin{align*}
\frac{d}{dt}\left\|\sqrt{\rho+\rho_{0}}\mathbf{V}_{t}\right\|_{L^2\left(\Omega\right)}^2+\left\|\nabla\mathbf{V}_{t}\right\|_{L^2\left(\Omega\right)}^2+\left\|\nabla^2\mathbf{V}\right\|_{L^2\left(\Omega\right)}^2
+\left\|\nabla P\right\|_{L^2\left(\Omega\right)}^2\leq C\left\|\sqrt{\rho+\rho_{0}}\mathbf{V}_{t}\right\|_{L^2\left(\Omega\right)}^2+C\delta_{0}^2.
\end{align*}
Then, from Gronwall inequality, one can deduce that 
\begin{align}\label{fromshangmian0102}
\int_{\Omega}\left(\rho+\rho_{0}\right)\mathbf{V}_{t}^2dxdy+\int_{0}^{t}\left(\left\|\mathbf{V}_{s}\right\|_{H^{1}\left(\Omega\right)}^2+\left\|\nabla^2\mathbf{V}\right\|_{L^2\left(\Omega\right)}^2\right)ds\leq C\delta_{0}^2.
\end{align} 
Again, from \eqref{buzhuoji0102}, we have 
\begin{align}\label{wbaobeiduibuqi0102}
\begin{aligned}
\left\|\nabla^2\mathbf{V}\right\|_{L^2\left(\Omega\right)}^2+\left\|\nabla P\right\|_{L^2\left(\Omega\right)}^2
\leq C\delta_{0}^2.
\end{aligned}
\end{align}

4)Estimations of \(\left\|\rho_{t}\right\|_{L^2\left(\Omega\right)}\) and \(\left\|\nabla \rho\right\|_{L^2\left(\Omega\right)}\). 

According to \eqref{fromshangmian0102} and \eqref{wbaobeiduibuqi0102}, we can deduce that 
\[
-\mu\Delta\mathbf{V}+\nabla P=-\left(\rho+\rho_{0}\right)\mathbf{V}_{t}-\left(\rho+\rho_{0}\right)\mathbf{V}\cdot\nabla\mathbf{V}-\rho\nabla f\in L^{2}\left(0,T; \left[L^4\left(\Omega\right)\right]^2\right),
\]
which together with Stokes' estimate gives that 
\begin{align}\label{fangshong0103}
\begin{aligned}
\int_{0}^{t}\left\|\nabla^2\mathbf{V}\right\|_{L^{4}\left(\Omega\right)}^2ds+\int_{0}^{t}\left\|\nabla P\right\|_{L^{4}\left(\Omega\right)}^2 ds \leq C\delta_{0}^2.
\end{aligned}
\end{align}
From \(\eqref{raodong1219}_{1}\), we obtain 
\begin{align}\label{jiushizheyang0103}
\begin{aligned}
&\rho_{tx}+\mathbf{V}_{x}\cdot\nabla\left(\rho+\rho_{0}\right)+\mathbf{V}\cdot\nabla\left(\rho_{x}+\rho_{0x}\right)=0,
\\
&\rho_{ty}+\mathbf{V}_{y}\cdot\nabla\left(\rho+\rho_{0}\right)+\mathbf{V}\cdot\nabla\left(\rho_{y}+\rho_{0y}\right)=0.
\end{aligned}
\end{align}
We multiply \(\eqref{jiushizheyang0103}_{1}\) and \(\eqref{jiushizheyang0103}_{2}\) by 
\(\rho_{x}\) and \(\rho_{y}\), respectively, then integrate them over \(\Omega\), add up these results and have 
\begin{align}\label{mingzi0103}
\begin{aligned}
\frac{d}{dt}\left\|\nabla \rho\right\|_{L^2\left(\Omega\right)}^2\leq 2\left(\left\|\nabla\mathbf{V}\right\|_{L^{\infty}\left(\Omega\right)}+1\right)\left\|\nabla\rho\right\|_{L^2\left(\Omega\right)}^2+C\delta_{0}^2,
\end{aligned}
\end{align}
where \(\left\|\nabla\mathbf{V}\right\|_{L^{\infty}\left(\Omega\right)}\) is integrable from \eqref{fangshong0103} and the Cauchy inequality is used. Furthermore, from the Gronwall inequality for \eqref{mingzi0103}, we have 
\begin{align}\label{tidu0103}
\left\|\nabla\rho\right\|_{L^2\left(\Omega\right)}^2\leq C\delta_{0}^2.
\end{align}

In addition, 
\[
\rho_{t}=-\mathbf{V}\cdot\nabla\left(\rho+\rho_{0}\right),
\]
which implies that 
\begin{align}\label{pinggu0103}
\left\|\rho_{t}\right\|_{L^2\left(\Omega\right)}^2\leq C\delta_{0}^2.
\end{align}

Thus, from \eqref{zaijian0102},\eqref{fromshangmian0102},\eqref{wbaobeiduibuqi0102}, \eqref{tidu0103} and \eqref{pinggu0103}, we can deduce \eqref{guji0102}. 
\end{proof}

Assume that 
\(\left\|\left(\mathbf{u}_{0},\theta_{0}\right)\right\|_{H^{2}\left(\Omega\right)}=\sqrt{\left\|\mathbf{u}_{0}\right\|_{H^2\left(\Omega\right)}^2+\left\|\theta_{0}\right\|_{H^{1}\left(\Omega\right)}^2}=\delta_{0}<\delta\), where \(\left(\mathbf{u}_{0},\theta_{0}\right)\) can be found in the Theorem \ref{xianxingbuwending1221} and \(\delta\) can be seen in Proposition \ref{nengliangguji0102}.
From Remark \ref{buweiling0202}, we can assume that 
\(u_{02}\neq 0\) and 
denote 
\[
\tau_{0}:=\frac{\left\|u_{02}\right\|_{L^2\left(\Omega\right)}}{\left\|\left(\mathbf{u}_{0},\theta_{0}\right)\right\|_{H^{2}\left(\Omega\right)}},
\]
then \(\tau_{0}\in (0,1]\).  Furthermore, let \(t_{K}=\frac{1}{\Lambda}\ln{\frac{2K}{\tau_{0}}}\). Thus, the solution \(\left(\mathbf{V},\rho\right)\) to the linearized system \eqref{xianxing1219} satisfies 
\begin{align}\label{xinan0103}
\left\|V_{2}\left(t_{K}\right)\right\|_{L^{2}\left(\Omega\right)}
=\left\|e^{t_{K}\Lambda}u_{02}\right\|_{L^{2}\left(\Omega\right)}=2K\delta_{0}.
\end{align}

Let 
\[
\left(\mathbf{u}_{0}^{\varepsilon},\theta_{0}^{\varepsilon}\right)
=\varepsilon\left(\mathbf{u}_{0},\theta_{0}\right),~\varepsilon\in(0,1).
\]
As \(\varepsilon\) is small enough, the maximum existence time \(T_{K\varepsilon}\) of the strong solution \(\left(\mathbf{V}^{\varepsilon},\rho^{\varepsilon}\right)\) to the nonlinear system \eqref{raodong1219} with initial value \(\left(\mathbf{u}_{0}^{\varepsilon},\theta_{0}^{\varepsilon}\right)\) can be greater than \(t_{K}\) and we have the following estimate 
\begin{align*}
\begin{aligned}
    \sup\limits_{t\in[0,T]} 
    \bigg{\{}
\left\|\mathbf{V}^{\varepsilon}\right\|_{H^2\left(\Omega\right)}^2
&+\left\|\rho^{\varepsilon}\right\|_{H^1\left(\Omega\right)}^2    +\left\|\nabla P^{\varepsilon}\right\|_{L^2\left(\Omega\right)}^2+
\left\|\left(\mathbf{V}_{t}^{\varepsilon},\rho_{t}^{\varepsilon}\right)\right\|_{L^2\left(\Omega\right)}^2
\\
&+\int_{0}^{t}\left[\left\|\mathbf{V}_{s}^{\varepsilon}\right\|_{H^1\left(\Omega\right)}^2+\left\|\nabla\mathbf{V}^{\varepsilon}\right\|_{H^1\left(\Omega\right)}^2\right]ds
    \bigg{\}}\leq C\delta_{0}^2,
    \end{aligned}
\end{align*}
where \(T<T_{K\varepsilon}\).

Finally, we utilize the method of contradiction to verify the Theorem \ref{lipuxizhi0202}.
\begin{lemma}\label{yinli0103}
For any positive and small enough \(\varepsilon\), the strong solution \(\left(\mathbf{V}^{\varepsilon},\rho^{\varepsilon}\right)\) of the nonlinear system \eqref{raodong1219} with the initial data \(\left(\mathbf{u}_{0}^{\varepsilon},\theta_{0}^{\varepsilon}\right)\) satisfies 
\begin{align*}
\left\|V_{2}^{\varepsilon}\left(t_{K}\right)\right\|_{L^2\left(\Omega\right)}\geq F\left(\left\|\left(\mathbf{u}_{0}^{\varepsilon},\theta_{0}^{\varepsilon}\right)\right\|_{H^{2}\left(\Omega\right)}\right),~t_{K}\in(0,T_{K\varepsilon}].
\end{align*}
\end{lemma}
\begin{proof}
This lemma is proved by the method of contradiction. We assume that there exists a positive constant \(\varepsilon_{0}\) such that 
\(\left(\mathbf{V}^{\varepsilon_{0}},\rho^{\varepsilon_{0}}\right)\) satisfies 
\begin{align*}
\left\|V_{2}^{\varepsilon}\left(t\right)\right\|_{L^2\left(\Omega\right)}\leq F\left(\left\|\left(\mathbf{u}_{0}^{\varepsilon_{0}},\theta_{0}^{\varepsilon_{0}}\right)\right\|_{H^{2}\left(\Omega\right)}\right),~\forall~t\in[0,t_{K}].
\end{align*}
Thus, by the property of \(F\), we can deduce that 
\begin{align}\label{yexu0103}
\left\|V_{2}^{\varepsilon_{0}}\left(t\right)\right\|_{L^{2}\left(\Omega\right)}\leq K\delta_{0}\varepsilon,~\forall~t\in[0,t_{K}].
\end{align}
Let \(\left(\overline{\mathbf{V}}^{\varepsilon_{0}},\overline{P}^{\varepsilon_{0}},\overline{\rho}^{\varepsilon_{0}}\right)=\frac{\left(\mathbf{V}^{\varepsilon_{0}},P^{\varepsilon_{0}},\rho^{\varepsilon_{0}}\right)}{\varepsilon_{0}}\). Then, we have 
\begin{align}\label{zhemeduo0103}
\begin{cases}
\overline{\rho}_{t}^{\varepsilon_{0}}+\varepsilon_{0}\overline{\mathbf{V}}^{\varepsilon_{0}}\cdot\nabla\overline{\rho}^{\varepsilon_{0}}+\overline{\mathbf{V}}^{\varepsilon_{0}}\cdot\nabla\rho_{0}=0,
\\
\left(\varepsilon_{0}\overline{\rho}_{t}^{\varepsilon_{0}}+\rho_{0}\right)\overline{\mathbf{V}}_{t}^{\varepsilon_{0}}+\left(\varepsilon_{0}\overline{\rho}^{\varepsilon_{0}}+\rho_{0}\right)\varepsilon_{0}\overline{\mathbf{V}}^{\varepsilon_{0}}\cdot\nabla\overline{\mathbf{V}}^{\varepsilon_{0}}+\nabla\overline{P}^{\varepsilon_{0}}=\mu\Delta\overline{\mathbf{V}}^{\varepsilon_{0}}-\overline{\rho}^{\varepsilon_{0}}\nabla f,
\\
\nabla\cdot\overline{\mathbf{V}}^{\varepsilon_{0}}=0,~\overline{\mathbf{V}}^{\varepsilon_{0}}|_{\partial\Omega}=\mathbf{0},~\left(\overline{\mathbf{V}}^{\varepsilon_{0}}\left(0\right),\overline{\rho}^{\varepsilon_{0}}\right)=\left(\mathbf{u}_{0},\theta_{0}\right).
\end{cases}
\end{align}
And we have 
\begin{align*}
\begin{aligned}
    \sup\limits_{t\in[0,T]} 
    \bigg{\{}
\left\|\mathbf{V}^{\varepsilon_{0}}\right\|_{H^2\left(\Omega\right)}^2
&+\left\|\rho^{\varepsilon_{0}}\right\|_{H^1\left(\Omega\right)}^2    +\left\|\nabla P^{\varepsilon_{0}}\right\|_{L^2\left(\Omega\right)}^2+
\left\|\left(\mathbf{V}_{t}^{\varepsilon_{0}},\rho_{t}^{\varepsilon_{0}}\right)\right\|_{L^2\left(\Omega\right)}^2
\\
&+\int_{0}^{t}\left[\left\|\mathbf{V}_{s}^{\varepsilon_{0}}\right\|_{H^1\left(\Omega\right)}^2+\left\|\nabla\mathbf{V}^{\varepsilon_{0}}\right\|_{H^1\left(\Omega\right)}^2\right]ds
    \bigg{\}}\leq C\delta_{0}^2,
    \end{aligned}
\end{align*}
where \(T<T_{K\varepsilon_{0}}\). This estimate allows us to obtain the following convergent subsequences (rebelled by the same notation), as \(\varepsilon_{0}\rightarrow 0^{+}\), 
\begin{align*}
\begin{aligned}
&\left(\overline{\mathbf{V}}^{\varepsilon_{0}}_{t},\nabla\overline{P}^{\varepsilon_{0}},\overline{\rho}_{t}^{\varepsilon_{0}}\right)\rightarrow 
\left(\overline{\mathbf{V}}_{t},\nabla\overline{P},\overline{\rho}_{t}\right)~\text{weakly~star~in}~L^{\infty}\left(0,t_{K};\left[L^{2}\left(\Omega\right)\right]^{5}\right),
\\
&\left(\overline{\mathbf{V}}^{\varepsilon_{0}},\overline{\rho}^{\varepsilon_{0}}\right)\rightarrow \left(\overline{\mathbf{V}},\overline{\rho}\right)~\text{weakly~star~in~}L^{\infty}\left(0,t_{K};\left[H^{2}\left(\Omega\right)\right]^2\times H^{1}\left(\Omega\right)\right),
\\
&\left(\overline{\mathbf{V}}^{\varepsilon_{0}},\overline{\rho}^{\varepsilon_{0}}\right)\rightarrow\left(\overline{\mathbf{V}},\overline{\rho}\right)~\text{strongly~in~}C\left(0,t_{K};[H^{1}\left(\Omega\right)]^2\times L^{2}\left(\Omega\right)\right).
\end{aligned}
\end{align*}
And from \eqref{yexu0103}, we have 
\begin{align}\label{xuexi0103}
\sup\limits_{0\leq t\leq t_{K}}
\left\|\overline{V}_{2}\left(t\right)\right\|_{L^{2}\left(\Omega\right)}
\leq K\delta_{0}.
\end{align}
In \eqref{zhemeduo0103}, let \(\varepsilon_{0}\rightarrow 0^{+}\), we can verify that \(\left(\overline{\mathbf{V}},\overline{\rho}\right)\) is the weak solution to the linearized system \eqref{xianxing1219} with the initial value \(\left(\overline{\mathbf{V}}\left(0\right),\overline{\rho}\left(0\right)\right)=\left(\mathbf{u}_{0},\theta_{0}\right)\). Then by the virtue of uniqueness of solution, see the Lemma \ref{weiyixing1230}, we have 
\[
\overline{\mathbf{V}}=\mathbf{V},~\overline{\rho}=\rho.
\]
However, from \eqref{xinan0103} and \eqref{xuexi0103}, it deduces a contradiction. Thus, the Lemma is proved. 
\end{proof}
\begin{remark}\label{dingyi0202}
    For the case \(u_{01}\neq 0\), the proof is similar. Thus, we omit the procedure.
\end{remark}
\subsection{Proof of the Theorem \ref{hadamardyiyixia0202}}\label{nonlinear0203}
The following is the nonlinear energy estimate of the strong solution to the system \eqref{raodong1219}, which plays a crucial role in the proof of Theorem \ref{hadamardyiyixia0202}.
\begin{proposition}\label{feixiannengliang0104}
    There exists a positive constant \(\delta\ll 1\) such that 
    when 
    \[
    \mathcal{E}\left(t\right):
    =\sqrt{\left\|\mathbf{V}\left(t\right)\right\|_{H^{1}\left(\Omega\right)}^2+\left\|\rho\left(t\right)\right\|_{L^2\left(\Omega\right)}^2}\leq \delta_{0}<\delta,~\forall~t\in [0,T]\subset [0,T_{\text{max}}),
    \]
    where 
    \(\left(\mathbf{V}\left(t\right),\rho\left(t\right)\right)\) is the strong solution to the nonlinear system \eqref{raodong1219} with initial value \(\left(\mathbf{V}\left(0\right),\rho\left(0\right)\right)\in \left[H^{2}\left(\Omega\right)\cap\mathbf{H}_{\text{div}}\right]\times \left[H^{1}\left(\Omega\right)\cap L^{\infty}\left(\Omega\right)\right] \) and \(T_{\text{max}}\) is the maximum existence time, we have the following estimate,
    \begin{align*}
    \widetilde{\mathcal{E}}^{2}\left(t\right)
    +\left\|\left(\mathbf{V}_{t},\nabla P\right)\right\|_{L^2\left(\Omega\right)}^2+\int_{0}^{t}\left\|\left(\nabla\mathbf{V},\mathbf{V}_{s},\nabla\mathbf{V}_{s}\right)\right\|_{L^2\left(\Omega\right)}^2 ds
    \leq C\left(\widetilde{\mathcal{E}}_{0}^{2}+\int_{0}^{t}\left\|\left(\mathbf{V},\rho\right)\right\|_{L^2(\Omega)}^2ds\right),
    \end{align*}
 where 
 \(\widetilde{\mathcal{E}}\left(t\right)=\sqrt{\left\|\mathbf{V}\left(t\right)\right\|_{H^{2}\left(\Omega\right)}^2+\left\|\rho\left(t\right)\right\|_{L^2\left(\Omega\right)}^2}\), \(\widetilde{\mathcal{E}}_{0}=\widetilde{\mathcal{E}}\left(0\right)\) and \(C=C\left(f,\rho_{0},\mu,\Omega\right)>0\).
\end{proposition}
\begin{proof}
Multiply \(\eqref{raodong1219}_{1}\) by \(\rho\) and integrate it over \(\Omega\), then we have 
\begin{align}\label{duibuqi0104}
\frac{d}{dt}\left\|\rho\right\|_{L^{2}\left(\Omega\right)}^2=-2\int_{\Omega}\mathbf{V}\cdot\nabla \rho_{0}\rho dxdy 
\leq C\left\|\left(\mathbf{V},\rho\right)\right\|_{L^{2}\left(\Omega\right)}^2.
\end{align}
After multiplying \(\eqref{raodong1219}_{1}\) by 
\(\left|\rho+\rho_{0}\right|^{q-2}\left(\rho+\rho_{0}\right)\) \(\left(q>1\right)\) and integrating it over \(\Omega\) by parts, we obtain 
\[
\left\|\rho+\rho_{0}\right\|_{L^{q}\left(\Omega\right)}=\left\|\rho\left(0\right)+\rho_{0}\right\|_{L^{q}\left(\Omega\right)},
\]
which shows that 
\begin{align*}
\left\|\rho\right\|_{L^{q}\left(\Omega\right)}\leq \left\|\rho\left(0\right)+\rho_{0}\right\|_{L^{q}\left(\Omega\right)}+\left\|\rho_{0}\right\|_{L^{q}\left(\Omega\right)},~\forall~1<q\leq\infty.
\end{align*}
Multiply \(\eqref{raodong1219}_{2}\) by \(\mathbf{V}\) and integrate it over \(\Omega\) by parts, we have 
\begin{align}\label{jiehe0104}
\frac{d}{dt}\int_{\Omega}\left(\rho+\rho_{0}\right)\left|\mathbf{V}\right|^2dxdy
+2\mu\left\|\nabla\mathbf{V}\right\|_{L^{2}\left(\Omega\right)}^2 
=-2\int_{\Omega}\rho\nabla f\cdot\mathbf{V}dxdy\leq C\left\|\left(\mathbf{V},\rho\right)\right\|_{L^{2}\left(\Omega\right)}^2.
\end{align}
Add \eqref{duibuqi0104} and \eqref{jiehe0104}, then integrate this result with respect to \(t\), we get 
\begin{align}\label{meiyouwenti0104}
\left\|\left(\mathbf{V},\rho\right)\right\|_{L^2(\Omega)}^2+\int_{0}^{t}\left\|\nabla\mathbf{V}\right\|_{L^{2}\left(\Omega\right)}^2ds
\leq C\left(\widetilde{\mathcal{E}}_{0}^2+\int_{0}^{t}\left\|\left(\mathbf{V},\rho\right)\right\|_{L^{2}\left(\Omega\right)}^2ds\right),
\end{align}
where the assumption \(\left(\rho\left(0\right)+\rho_{0}\right)>\sigma\) is used.

Multiply \(\eqref{raodong1219}_{2}\) by \(\mathbf{V}_{t}\) and integrate it over \(\Omega\), we obtain 
\begin{align}\label{yongxin0104}
\int_{\Omega}\left(\rho+\rho_{0}\right)\left|\mathbf{V}_{t}\right|^2dx dy=&-\int_{\Omega}\left(\rho+\rho_{0}\right)\mathbf{V}\cdot\nabla\mathbf{V}\cdot\mathbf{V}_{t}dxdy
\\
&+\mu\int_{\Omega}\Delta\mathbf{V}\cdot\mathbf{V}_{t}dxdy-\int_{\Omega}\rho\nabla f\cdot\mathbf{V}_{t}dxdy,
\end{align}
which together with Cauchy inequality and letting \(t\rightarrow 0^{+}\) yields that 
\begin{align}\label{zenmekeneng0104}
\lim\limits_{t\rightarrow 0^{+}}
\int_{\Omega}\left(\rho+\rho_{0}\right)\left|\mathbf{V}_{t}\right|^{2}dxdy\leq C\widetilde{\mathcal{E}}_{0}^{2}.
\end{align}

We differentiate \(\eqref{raodong1219}_{2}\) with respect to \(t\), multiply it by \(\mathbf{V}_{t}\) and integrate it over \(\Omega\), then obtain 
\begin{align}\label{hongceng0104}
\begin{aligned}
\frac{1}{2}\frac{d}{dt}&\int_{\Omega}
\left(\rho+\rho_{0}\right)\left|\mathbf{V}_{t}\right|^2dxdy 
+\mu\left\|\nabla\mathbf{V}_{t}\right\|_{L^{2}\left(\Omega\right)}^2
=\sum\limits_{i=1}^{4}\widetilde{K}_{i},
\end{aligned}
\end{align}
where 
\begin{align*}
&\widetilde{K}_{1}=-2\int_{\Omega}\left(\rho+\rho_{0}\right)\mathbf{V}\cdot\nabla\mathbf{V}_{t}\cdot\mathbf{V}_{t}dxdy,~ 
\widetilde{K}_{2}=-\int_{\Omega}\left(\rho+\rho_{0}\right)\mathbf{V}\cdot\nabla\left(\mathbf{V}\cdot\nabla\mathbf{V}\cdot\mathbf{V}_{t}\right)dxdy,
\\
&\widetilde{K}_{3}=-\int_{\Omega}\left(\rho+\rho_{0}\right)\mathbf{V}_{t}\cdot\nabla\mathbf{V}\cdot\mathbf{V}_{t}dxdy,~
\widetilde{K}_{4}=-\int_{\Omega}\left(\rho+\rho_{0}\right)\mathbf{V}\cdot\nabla\left(\nabla f\cdot\mathbf{V}_{t}\right)dxdy.
\end{align*}
Apply H\"{o}lder inequality, Cauchy inequality and Lemma \ref{poincarebudengshi0201} to \(\widetilde{K}_{i}\) (\(i=1,2,3,4\)), we can get 
\begin{align*}
\begin{aligned}
&\left|\widetilde{K}_{1}\right|
\leq C\left\|\mathbf{V}\right\|_{L^{4}\left(\Omega\right)}^2\left\|\nabla\mathbf{V}_{t}\right\|_{L^{2}\left(\Omega\right)}^2+\varepsilon \left\|\nabla\mathbf{V}_{t}\right\|_{L^{2}\left(\Omega\right)}^2,~\left|\widetilde{K}_{3}\right|\leq C\left\|\nabla\mathbf{V}\right\|_{L^{2}\left(\Omega\right)}\left\|\nabla\mathbf{V}_{t}\right\|_{L^{2}\left(\Omega\right)}^{2}
\\
&\begin{aligned}
\left|\widetilde{K}_{2}\right|
\leq &C\bigg{[}\left\|\mathbf{V}\right\|_{L^{4}\left(\Omega\right)}\left\|\nabla\mathbf{V}\right\|_{L^{4}\left(\Omega\right)}^2\left\|\mathbf{V}_{t}\right\|_{L^{4}\left(\Omega\right)}
\\
&+\left\|\mathbf{V}\right\|_{L^{8}\left(\Omega\right)}^2\left\|\nabla^2\mathbf{V}\right\|_{L^2\left(\Omega\right)}\left\|\mathbf{V}_{t}\right\|_{L^{4}\left(\Omega\right)}
+\left\|\mathbf{V}\right\|_{L^{8}\left(\Omega\right)}^2\left\|\nabla\mathbf{V}\right\|_{L^{4}\left(\Omega\right)}\left\|\nabla\mathbf{V}_{t}\right\|_{L^{2}\left(\Omega\right)}
\bigg{]}
\\
&\leq C\left\|\nabla\mathbf{V}\right\|_{L^{2}\left(\Omega\right)}^{4}\left\|\nabla^2\mathbf{V}\right\|_{L^{2}\left(\Omega\right)}^2+\varepsilon\left\|\nabla\mathbf{V}_{t}\right\|_{L^{2}\left(\Omega\right)}^2,
\end{aligned}
\\
&\left|\widetilde{K}_{4}\right|
\leq C\left\|\mathbf{V}\right\|_{L^2(\Omega)}^{2}+\varepsilon\left\|\nabla\mathbf{V}_{t}\right\|_{L^{2}\left(\Omega\right)}^2,
\end{aligned}
\end{align*}
where Lemmas \ref{poincarebudengshi0201}-\ref{zuidazhi0201} are used. Since \(\mathcal{E}\left(t\right)\ll 1\), then from \eqref{hongceng0104}, we have 
\begin{align}\label{yonggan0104}
\frac{d}{dt}\int_{\Omega}
\left(\rho+\rho_{0}\right)\left|\mathbf{V}_{t}\right|^2dxdy 
+\left\|\nabla\mathbf{V}_{t}\right\|_{L^{2}\left(\Omega\right)}^2\leq C_{\delta_{0}}
\left\|\nabla^{2}\mathbf{V}\right\|_{L^{2}\left(\Omega\right)}^2+C\left\|\mathbf{V}\right\|_{L^{2}\left(\Omega\right)}^2,
\end{align}
where \(C_{\delta_{0}}>0\) is small enough.

In addition, by virtue of the Stokes' estimate and \(\mathcal{E}\left(t\right)\ll 1\), then from \(\eqref{raodong1219}_{2}\), we obtain 
\begin{align}\label{tianxiajiashi0104}
\begin{aligned}
\left\|\nabla^2\mathbf{V}\right\|_{L^{2}\left(\Omega\right)}^{2}
+\left\|\nabla P\right\|_{L^{2}\left(\Omega\right)}^2\leq C\left(\left\|\mathbf{V}_{t}\right\|_{L^{2}\left(\Omega\right)}^2+\left\|\rho\right\|_{L^{2}\left(\Omega\right)}^{2}\right)
\end{aligned}
\end{align}

From \eqref{yonggan0104} and \eqref{tianxiajiashi0104}, we have 
\[
\frac{d}{dt}\int_{\Omega}
\left(\rho+\rho_{0}\right)\left|\mathbf{V}_{t}\right|^2dxdy+\left\|\nabla^2\mathbf{V}\right\|_{L^{2}\left(\Omega\right)}^{2}
+\left\|\nabla\mathbf{V}_{t}\right\|_{L^{2}\left(\Omega\right)}^2\leq C\left\|\left(\mathbf{V},\rho\right)\right\|_{L^{2}\left(\Omega\right)}^2,
\]
which together with \eqref{zenmekeneng0104} yields that 
\begin{align}\label{zhenshiwoxiang0104}
\left\|\mathbf{V}_{t}\right\|_{L^{2}\left(\Omega\right)}^2+
\int_{0}^{t}\left\|\left(\nabla^2\mathbf{V},\nabla\mathbf{V}_{t}\right)\right\|_{L^{2}\left(\Omega\right)}^2ds\leq C\left(\widetilde{\mathcal{E}}_{0}^{2}+\int_{0}^{t}\left\|\left(\mathbf{V},\rho\right)\right\|_{L^2(\Omega)}^2ds\right).
\end{align}

From \eqref{meiyouwenti0104}, \eqref{tianxiajiashi0104} and \eqref{zhenshiwoxiang0104}, we can obtain the estimate desired. 

\end{proof}
\begin{proof}\textbf{The proof of Theorem \ref{hadamardyiyixia0202}.}
For convenience, we explain that the constants \(C_{i}\) (\(i=1,\cdots,15\)) in the following proof are independent of \(\delta^{*}\) in advance.
Choose \(\left(\mathbf{u}_{0},\theta_{0}\right)\) satisfies the relationship in Remark \ref{buweiling0202}. Let 
\[m_{0}=\min\left\{\left\|u_{01}\right\|_{L^1\left(\Omega\right)},\left\|u_{02}\right\|_{L^1\left(\Omega\right)},\left\|\theta_{0}\right\|_{L^1\left(\Omega\right)}\right\},\] then \(m_{0}>0\).
Furthermore, we can assume that 
\[
\left\|\mathbf{u}_{0}\right\|_{H^{2}\left(\Omega\right)}^2+\left\|\theta_{0}\right\|_{H^{1}\left(\Omega\right)}^2=1.
\]

Denote 
\[
\left(\mathbf{u}_{0}^{\delta^*},\theta_{0}^{\delta^*}\right)=\delta^{*}
\left(\mathbf{u}_{0},\theta_{0}\right),~C_{1}=\left\|\left(\mathbf{u}_{0},\theta_{0}\right)\right\|_{L^{2}\left(\Omega\right)},~C_{2}=\left\|\left(\mathbf{u}_{0},\theta_{0}\right)\right\|_{L^{1}\left(\Omega\right)},~\delta^{*}<\delta,
\]
where \(\delta>0\) is defined in Proposition \ref{feixiannengliang0104}, \(\left\|\left(\mathbf{u}_{0},\theta_{0}\right)\right\|_{L^{2}\left(\Omega\right)}:=\left\|\theta_{0}\right\|_{L^{2}\left(\Omega\right)}+\left\|\mathbf{u}_{0}\right\|_{L^{2}\left(\Omega\right)}\), 
\(\left\|\left(\mathbf{u}_{0},\theta_{0}\right)\right\|_{L^{1}\left(\Omega\right)}:=\left\|\theta_{0}\right\|_{L^{1}\left(\Omega\right)}+\left\|\mathbf{u}_{0}\right\|_{L^{1}\left(\Omega\right)}\). Since \(\delta^{*}\) is small enough and \(\rho_{0}>0\) indicating that \(\rho_{0}+\theta_{0}^{\delta^{*}}>0\), then according to Lemma \ref{shidingxingwenti0202}, we can assume that
\(\left(\mathbf{V}^{\delta^{*}},\rho^{\delta^{*}}\right)\) is the strong solution to the nonlinear system \eqref{raodong1219} with the initial value \(\left(\mathbf{u}_{0}^{\delta^{*}},\theta_{0}^{\delta^{*}}\right)\). From Lemma \ref{shidingxingwenti0202}, 
\[
\left(\mathbf{V}^{\delta^{*}},\rho^{\delta^{*}}\right)\in C\left(0,T;\mathbf{H}_{\text{div}}\times L^{2}\left(\Omega\right)\right).
\]
Define 
\begin{align*}
\begin{aligned}
T^{*}=\sup\left\{t\in (0,+\infty)|\sqrt{\left\|\mathbf{V}^{\delta^{*}}\left(s\right)\right\|_{H^{1}\left(\Omega\right)}^2+\left\|\rho^{\delta^{*}}\left(s\right)\right\|_{L^{2}\left(\Omega\right)}^2},~\forall~s\leq t\right\},
\\
T^{**}=\sup\left\{t\in (0,+\infty)|\left\|\left(\mathbf{V}^{\delta^{*}},\rho^{\delta^{*}}\right)\left(s\right)\right\|_{L^{2}\left(\Omega\right)}\leq 2\delta^{*}C_{1}e^{\Lambda s},~\forall~s\leq t \right\}.
\end{aligned}
\end{align*}
Since \(\delta^{*}<\delta\) and 
\(\left(\mathbf{V}^{\delta^{*}},\rho^{\delta^{*}}\right)\in C\left(0,T;\mathbf{H}_{\text{div}}\times L^{2}\left(\Omega\right)\right)\), then \(T^{*}T^{**}>0\). Furthermore, it is not difficult to verify that 
\begin{align}\label{wuquduoxiang0104}
\begin{aligned}
&\text{if}~T^{*}<\infty,~\sqrt{\left\|\mathbf{V}^{\delta^{*}}\left(T^{*}\right)\right\|_{H^{1}\left(\Omega\right)}^2+\left\|\rho^{\delta^{*}}\left(T^{*}\right)\right\|_{L^{2}\left(\Omega\right)}^2}=\delta_{0},
\\
&\text{if}~T^{**}<T_{\text{max}}, 
\left\|\left(\mathbf{V}^{\delta^{*}},\rho^{\delta^{*}}\right)\left(T^{**}\right)\right\|_{L^{2}\left(\Omega\right)}=2\delta^{*}C_{1}e^{\Lambda T^{**}}.
\end{aligned}
\end{align}
And from the Proposition \ref{feixiannengliang0104}, \(\forall ~t\leq \min\left\{T^{*},T^{**}\right\}\), we have 
\begin{align}\label{xuyaonini0210}
\begin{aligned}
\left\|\mathbf{V}^{\delta^{*}}\right\|_{H^{2}\left(\Omega\right)}^2
+\left\|\rho^{\delta^{*}}\right\|_{L^{2}\left(\Omega\right)}^2
+\left\|\mathbf{V}_{t}^{\delta^{*}}\right\|_{L^{2}\left(\Omega\right)}^2+\int_{0}^{t}\left\|\mathbf{V}^{\delta^{*}}_{s}\right\|_{H^{1}\left(\Omega\right)}^2ds
\leq C_{3}\left(\delta^{*}\right)^2 e^{2\Lambda t}.
\end{aligned}
\end{align}
% where \(C_{3}>0\) is independent of \(\delta^{*}\).

Let \(\left(\mathbf{V}^{d},\rho^{d}\right)=\left(\mathbf{V}^{\delta^{*}},\rho^{\delta^{*}}\right)-\delta^{*}\left(e^{\Lambda t}\mathbf{u}_{0},e^{\Lambda t}\theta_{0}\right)=\left(\mathbf{V}^{\delta^{*}},\rho^{\delta^{*}}\right)-\left(\mathbf{V}^{a},\rho^{a}\right)\), where \(\left(\mathbf{V}^{\delta^{*}},\rho^{\delta^{*}}\right)\) and \(\left(\mathbf{V}^{a},\rho^{a}\right)\) are the strong solutions to the nonlinear system \eqref{raodong1219} with initial value \(\left(\mathbf{u}_{0}^{\delta^{*}},\theta_{0}^{\delta^{*}}\right)\) and to the linearized system \eqref{xianxing1219} with the same initial value, respectively. Then, we have 
\begin{align}\label{suizaiyi0104}
\begin{cases}
\rho_{t}^{a}+\mathbf{V}^{a}\cdot\nabla\rho_{0}=0,
\\
\rho_{0}\mathbf{V}_{t}^{a}+\nabla P^{a}-\mu\Delta\mathbf{V}^{a}+\rho^{a}\nabla f=0,
\\
\nabla\cdot\mathbf{V}^{a}=0,~
\mathbf{V}^{a}|_{\partial\Omega}=\mathbf{0},~\left(\mathbf{V}^{a},\rho^{a}\right)\left(0\right)=\left(\mathbf{u}_{0}^{\delta^{*}},\theta_{0}^{\delta^{*}}\right),
\end{cases}
\end{align} 
and
\begin{align}\label{suizaiyi0210}
\begin{cases}
\rho_{t}^{\delta^{*}}+\mathbf{V}^{\delta^{*}}\cdot\nabla\rho_{0}=-\mathbf{V}^{\delta^{*}}\cdot\nabla\rho^{\delta^{*}},
\\
\left(\rho^{\delta^{*}}+\rho_{0}\right)\mathbf{V}_{t}^{\delta^{*}}+\nabla P^{\delta^{*}}-\mu\Delta\mathbf{V}^{\delta^{*}}+\rho^{\delta^{*}}\nabla f=-\left(\rho^{\delta^{*}}+\rho_{0}\right)\mathbf{V}^{\delta^{*}}\cdot\nabla\mathbf{V}^{\delta^{*}},
\\
\nabla\cdot\mathbf{V}^{\delta^{*}}=0,~
\mathbf{V}^{\delta^{*}}|_{\partial\Omega}=\mathbf{0},~\left(\mathbf{V}^{\delta^{*}},\rho^{\delta^{*}}\right)\left(0\right)=\left(\mathbf{u}_{0}^{\delta^{*}},\theta_{0}^{\delta^{*}}\right).
\end{cases}
\end{align}
Let \(\widetilde{\rho}^{\delta^{*}}=\left(\rho^{\delta^{*}}+\rho_{0}\right)\), then \(\widetilde{\rho}^{\delta^{*}}>0\) since \(\rho_{0}+\delta^{*}\theta_{0}>0\). And from \eqref{suizaiyi0104} and \eqref{suizaiyi0210}, we can deduce 
\begin{align}\label{nijiushi0210}
\widetilde{\rho}^{\delta^{*}}
\mathbf{V}_{t}^{d}+\rho^{\delta^{*}}\mathbf{V}_{t}^{a}-\mu\Delta\mathbf{V}^{d}+\nabla P^{d}+\rho^{d}\nabla f=-\widetilde{\rho}^{\delta^{*}}
\mathbf{V}^{\delta^{*}}\cdot\nabla\mathbf{V}^{\delta^{*}}.
\end{align}
Differentiate \eqref{nijiushi0210} with the respect to \(t\), multiply the result by \(\mathbf{V}_{t}^{d}\) and integrate it over \(\Omega\), then by a series of calculation we obtain 
\begin{align}\label{xuyaonia0210}
\begin{aligned}
\frac{d}{dt}\left\|\sqrt{\widetilde{\rho}^{\delta^{*}}}\mathbf{V}_{t}^{d}\right\|_{L^2\left(\Omega\right)}^2+
2\mu\left\|\nabla\mathbf{V}_{t}^{d}\right\|_{L^2\left(\Omega\right)}^2-\frac{d}{dt}\int_{\Omega}h\left|\mathbf{V}^{d}\cdot\nabla f\right|^2dxdy=\sum\limits_{i=1}^{6}Z_{i},
\end{aligned}
\end{align}
where 
\begin{align*}
\begin{aligned}
&
Z_{1}=2\int_{\Omega}\mathbf{V}^{\delta^{*}}\cdot\nabla\rho^{\delta^{*}}\nabla f\cdot\mathbf{V}_{t}^{d}dxdy,~
Z_{2}=-2\int_{\Omega}\rho_{t}^{\delta^{*}}\mathbf{V}_{t}^{a}\cdot\mathbf{V}_{t}^{d}dxdy,
\\
&Z_{3}=-2\int_{\Omega}\rho^{\delta^{*}}\mathbf{V}_{tt}^{a}\cdot\mathbf{V}_{t}^{d}dxdy,~
Z_{4}=-2\int_{\Omega}\widetilde{\rho}^{\delta^{*}}_{t}\mathbf{V}^{\delta^{*}}\cdot\nabla\mathbf{V}^{\delta^{*}}\cdot\mathbf{V}_{t}^{d}dxdy,
\\
&Z_{5}=-2\int_{\Omega}\widetilde{\rho}^{\delta^{*}}
\mathbf{V}_{t}^{\delta^{*}}\cdot\nabla\mathbf{V}^{\delta^{*}}\cdot\mathbf{V}_{t}^{d}dxdy,~
Z_{6}=-4\int_{\Omega}\widetilde{\rho}^{\delta^{*}}
\mathbf{V}^{\delta^{*}}\cdot\nabla\mathbf{V}_{t}^{\delta^{*}}\cdot\mathbf{V}_{t}^{d}dxdy.
\end{aligned}
\end{align*}
Then utilize \eqref{xuyaonini0210} and \(\mathbf{V}^{d}=\mathbf{V}^{\delta^{*}}-\mathbf{V}^{a}\), we have the following estimates 
\begin{align}\label{yiyi0210}
\begin{aligned}
&\begin{aligned}
Z_{1}&\leq \left\|\mathbf{V}^{\delta^{*}}\right\|_{L^{\infty}\left(\Omega\right)}\left\|\rho^{\delta^{*}}\right\|_{L^2\left(\Omega\right)}
\left[\left\|\nabla^2 f\right\|_{L^{\infty}\left(\Omega\right)}
\left\|\mathbf{V}_{t}^{d}\right\|_{L^2\left(\Omega\right)}+\left\|\nabla f\right\|_{L^{\infty}\left(\Omega\right)}
\left\|\nabla\mathbf{V}_{t}^{d}\right\|_{L^2\left(\Omega\right)}\right]
\\
&\leq C_{4}\left[\left(\delta^{*}\right)^3
e^{3\Lambda t}+\delta^{*}e^{\Lambda t}\left\|\nabla\mathbf{V}_{t}^{\delta^{*}}\right\|_{L^2\left(\Omega\right)}^2
\right],
\end{aligned}
\\
&\begin{aligned}
Z_{2}&\leq \left\|\mathbf{V}^{\delta^{*}}\right\|_{L^{\infty}\left(\Omega\right)}\left\|\widetilde{\rho}^{\delta^{*}}\right\|_{L^\infty\left(\Omega\right)}
\left[\left\|\nabla\mathbf{V}_{t}^{a}\right\|_{L^{2}\left(\Omega\right)}
\left\|\mathbf{V}_{t}^{d}\right\|_{L^2\left(\Omega\right)}+\left\|\mathbf{V}_{t}^{a}\right\|_{L^{2}\left(\Omega\right)}
\left\|\nabla\mathbf{V}_{t}^{d}\right\|_{L^2\left(\Omega\right)}\right]
\\
&\leq C_{5}\left[\left(\delta^{*}\right)^3
e^{3\Lambda t}+\delta^{*}e^{\Lambda t}\left\|\nabla\mathbf{V}_{t}^{\delta^{*}}\right\|_{L^2\left(\Omega\right)}^2
\right],
\end{aligned}
\\
&\begin{aligned}
Z_{3}\leq C_{6}\left(\delta^{*}\right)^3e^{3\Lambda t},
\end{aligned}
\\
&\begin{aligned}
Z_{4}&\leq \left\|\mathbf{V}^{\delta^{*}}\right\|_{L^{\infty}}
\left\|\widetilde{\rho}^{\delta^{*}}\right\|_{L^{\infty}\left(\Omega\right)}\left\|\mathbf{V}^{\delta^{*}}\right\|_{H^2\left(\Omega\right)}^2\left[2
\left\|\mathbf{V}_{t}^{d}\right\|_{L^2\left(\Omega\right)}+\left\|\nabla\mathbf{V}_{t}^{d}\right\|_{L^2\left(\Omega\right)}
\right]
\\
&\leq C_{7}\left[\left(\delta^{*}\right)^3
e^{3\Lambda t}+\delta^{*}e^{\Lambda t}\left\|\nabla\mathbf{V}_{t}^{\delta^{*}}\right\|_{L^2\left(\Omega\right)}^2+\left(\delta^{*}\right)^4
e^{4\Lambda t}
\right],
\end{aligned}
\end{aligned}
\end{align}
\begin{align}\label{yiyi0310}
\begin{aligned}
&\begin{aligned}
Z_{5}&\leq \left\|\widetilde{\rho}^{\delta^{*}}\right\|_{L^{\infty}\left(\Omega\right)}\left\|\mathbf{V}_{t}^{\delta^{*}}\right\|_{L^2\left(\Omega\right)}\left\|\mathbf{V}^{\delta^{*}}\right\|_{H^2\left(\Omega\right)}\left\|\nabla\mathbf{V}_{t}^{d}\right\|_{L^2\left(\Omega\right)}
\\
&\leq C_{8}\left[\left(\delta^{*}\right)^3
e^{3\Lambda t}+\delta^{*}e^{\Lambda t}\left\|\nabla\mathbf{V}_{t}^{\delta^{*}}\right\|_{L^2\left(\Omega\right)}^2
\right],
\end{aligned}
\\
&\begin{aligned}
Z_{6}&\leq \left\|\widetilde{\rho}^{\delta^{*}}\right\|_{L^{\infty}\left(\Omega\right)}\left\|\mathbf{V}^{\delta^{*}}\right\|_{L^{\infty}
\left(\Omega\right)}\left\|\nabla\mathbf{V}_{t}^{\delta^{*}}\right\|_{L^2\left(\Omega\right)}
\left\|\mathbf{V}_{t}^{d}\right\|_{L^2\left(\Omega\right)}
\\
&\leq C_{9}\left[\left(\delta^{*}\right)^3
e^{3\Lambda t}+\delta^{*}e^{\Lambda t}\left\|\nabla\mathbf{V}_{t}^{\delta^{*}}\right\|_{L^2\left(\Omega\right)}^2
\right].
\end{aligned}
\end{aligned}
\end{align}
Furthermore, multiply \eqref{nijiushi0210} by \(\mathbf{V}_{t}^{d}\) and integrate the result over \(\Omega\), utilize \(\mathbf{V}^{d}\left(0\right)=\mathbf{0}\) and \(\rho^{d}\left(0\right)=0\), one can deduce that 
\begin{align}\label{chushizhi0210}
\left\|\sqrt{\widetilde{\rho}^{\delta^{*}}}\mathbf{V}_{t}^{d}\right\|_{L^2\left(\Omega\right)}^2|_{t=0}\leq C_{10}\left(\delta^{*}\right)^3.
\end{align}
Put \eqref{yiyi0210} and \eqref{yiyi0310} into \eqref{xuyaonia0210} and integrate the result from \(0\) to \(t\), utilize the conditions \eqref{yiyi0210}, \eqref{chushizhi0210},\(\mathbf{V}^{d}\left(0\right)=\mathbf{0}\) and \(\rho^{d}\left(0\right)=0\), we have 
\begin{align}\label{ali0210}
\begin{aligned}
\left\|\sqrt{\widetilde{\rho}^{\delta^{*}}}\mathbf{V}_{t}^{d}\right\|_{L^2\left(\Omega\right)}^2
&+2\mu\int_{0}^{t}\left\|\nabla\mathbf{V}_{s}^{d}\right\|_{L^2\left(\Omega\right)}^2 ds-\int_{\Omega}h\left|\mathbf{V}^{d}\cdot\nabla f\right|^2dxdy
\\
&\leq C_{11}\left(\delta^{*}\right)^3
e^{3\Lambda t}\left(1+\delta^{*}e^{\Lambda t}\right).
\end{aligned}
\end{align}

In addition, from Remark \eqref{xuyaode0210}, we have 
\begin{align*}
\begin{aligned}
-\int_{\Omega}h\left|\mathbf{V}^{d}\cdot\nabla f\right|^2dxdy 
\geq -\Lambda\mu\left\|\nabla\mathbf{V}^{d}\right\|_{L^2\left(\Omega\right)}^2-\Lambda^2 \left\|\sqrt{\rho_{0}}\mathbf{V}^{d}\right\|_{L^2\left(\Omega\right)}^2,
\end{aligned}
\end{align*}
which together with \eqref{ali0210} and \(\widetilde{\rho}^{\delta^{*}}= \rho_{0}+\rho^{\delta^{*}}\) yields that 
\begin{align}\label{dada0210}
\begin{aligned}
\left\|\sqrt{\rho_{0}}\mathbf{V}_{t}^{d}\right\|_{L^2\left(\Omega\right)}^2
&+2\mu\int_{0}^{t}\left\|\nabla\mathbf{V}_{s}^{d}\right\|_{L^2\left(\Omega\right)}^2 ds-\Lambda\mu\left\|\nabla\mathbf{V}^{d}\right\|_{L^2\left(\Omega\right)}^2-\Lambda^2 \left\|\sqrt{\rho_{0}}\mathbf{V}^{d}\right\|_{L^2\left(\Omega\right)}^2
\\
&\leq C_{12}\delta^{*}
e^{\Lambda t}\left(\left(\delta^{*}\right)^2
e^{2\Lambda t}+\left\|\nabla\mathbf{V}_{t}^{d}\right\|_{L^2\left(\Omega\right)}^2+\left(\delta^{*}\right)^3e^{3\Lambda t}\right).
\end{aligned}
\end{align}
A simple computation with Cauchy inequality shows that 
\begin{align*}
\begin{aligned}
&\Lambda \frac{d}{dt}\left\|\sqrt{\rho_{0}}\mathbf{V}^{d}\right\|_{L^2\left(\Omega\right)}^2=2\Lambda\int_{\Omega}
\rho_{0}\mathbf{V}_{t}^{d}\cdot\mathbf{V}^{d}dxdy\leq \Lambda^2\left\|\sqrt{\rho_{0}}\mathbf{V}^{d}\right\|_{L^2\left(\Omega\right)}^2+\left\|\sqrt{\rho_{0}}\mathbf{V}^{d}_{t}\right\|_{L^2\left(\Omega\right)}^2,
\\
&\mu\Lambda\left\|\nabla\mathbf{V}^{d}\right\|_{L^{2}\left(\Omega\right)}^2
=\mu\Lambda\int_{0}^{t}\frac{d}{ds}\left\|\nabla\mathbf{V}^{d}\right\|_{L^2\left(\Omega\right)}^2ds\leq \mu\Lambda^2\int_{0}^{t}
\left\|\nabla\mathbf{V}^{d}\right\|_{L^2\left(\Omega\right)}^2ds+\mu\int_{0}^{t}\left\|\nabla\mathbf{V}_{s}^{d}\right\|_{L^2\left(\Omega\right)}^2ds,
\end{aligned}
\end{align*}
Thus, from \eqref{dada0210}, we have 
\begin{align*}
\begin{aligned}
\mu\Lambda \left\|\nabla\mathbf{V}^{d}\right\|_{L^2\left(\Omega\right)}^2+\Lambda\frac{d}{dt}
\left\|\sqrt{\rho_{0}}\mathbf{V}^{d}\right\|_{L^2\left(\Omega\right)}^2
\leq &C_{12}\delta^{*}
e^{\Lambda t}\left(\left(\delta^{*}\right)^2
e^{2\Lambda t}+\left\|\nabla\mathbf{V}_{t}^{d}\right\|_{L^2\left(\Omega\right)}^2+\left(\delta^{*}\right)^3e^{3\Lambda t}\right)
\\
&+2\Lambda^2\left\|\sqrt{\rho_{0}}\mathbf{V}^{d}\right\|_{L^2\left(\Omega\right)}^2+2\mu\Lambda^2\int_{0}^{t}\left\|\nabla\mathbf{V}^{d}\right\|_{L^2\left(\Omega\right)}^2ds,
\end{aligned}
\end{align*}
which together with Gronwall's inequality, \eqref{xuyaonini0210} and \(\mathbf{V}^{d}\left(0\right)=\mathbf{0}\) yields that 
\begin{align}\label{houmian0210}
\left\|\mathbf{V}^{d}\right\|_{L^2\left(\Omega\right)}\leq C_{13}\left[\left(\delta^{*}\right)^{\frac{3}{2}}e^{\frac{3\Lambda}{2}t}+\left(\delta^{*}\right)^2e^{2\Lambda t}\right].
\end{align}
From \(\eqref{suizaiyi0104}_{1}\) and \(\eqref{suizaiyi0210}_{1}\), we have 
\begin{align}\label{henshuai0210}
\rho_{t}^{d}=-\mathbf{V}^{d}\cdot\nabla\rho_{0}
-\mathbf{V}^{\delta^{*}}\cdot\nabla \rho^{d}-\mathbf{V}^{\delta^{*}}\nabla \rho^{a}.
\end{align}
Multiply \eqref{henshuai0210} by \(\rho^{d}\), integrate the result over \(\Omega\) and from \eqref{houmian0210} and \(\rho^{d}\left(0\right)=0\), we can obtain
\begin{align*}
\left\|\rho^{d}\right\|_{L^2\left(\Omega\right)}
\leq C_{14}\left[\left(\delta^{*}\right)^{\frac{3}{2}}e^{\frac{3\Lambda}{2}t}+\left(\delta^{*}\right)^2e^{2\Lambda t}\right].
\end{align*}
Thus, we have 
\begin{align}\label{pingtai0104}
\left\|\left(\mathbf{V}^{d},\rho^{d}\right)\right\|_{L^2\left(\Omega\right)}\leq C_{15}\left[\left(\delta^{*}\right)^{\frac{3}{2}}e^{\frac{3\Lambda}{2}t}+\left(\delta^{*}\right)^2e^{2\Lambda t}\right].
\end{align}

Let \(\varepsilon_{0}=\min\left\{\frac{\delta_{0}}{2\sqrt{C_{3}}}, \xi,\frac{m_{0}^2}{32\left|\Omega\right|C_{15}^{2}}\right\}\), where 
\(\left(2\xi\right)^{\frac{1}{2}}+2\xi=\frac{C_{1}}{2C_{15}}\), and \(T^{\delta^{*}}=\frac{\ln{\frac{2\varepsilon_{0}}{\delta^{*}}}}{\Lambda}\). Next, we show that \(T^{\delta^{*}}=\min\bigg{\{}T^{\delta^{*}},T^{*},T^{**}\bigg{\}}\). Indeed, 
if \(T^{*}=\min\left\{T^{\delta^{*}},T^{*},T^{**}\right\}\), then \(T^{*}<\infty\). Thus, 
\[
\sqrt{\left\|\mathbf{V}^{\delta^{*}}\left(T^{*}\right)\right\|_{H^{1}\left(\Omega\right)}^2+\left\|\rho^{\delta^{*}}\left(T^{*}\right)\right\|_{L^{2}\left(\Omega\right)}^2}\leq \sqrt{C_{3}}\delta^{*}e^{\Lambda T^{*}}\leq \sqrt{C_{3}}\delta^{*}e^{\Lambda T^{\delta^{*}}}=2\varepsilon_{0}\sqrt{C_{3}}<\delta_{0},
\]
which contradicts \(\eqref{wuquduoxiang0104}_{1}\); 
if \(T^{**}=\min\left\{T^{\delta^{*}},T^{*},T^{**}\right\}\), then \(T^{**}<T_{\text{max}}\). Thus, 
\begin{align*}
\begin{aligned}
\left\|\left(\mathbf{V}^{\delta^{*}},\rho^{\delta^{*}}\right)\left(T^{**}\right)\right\|_{L^{2}\left(\Omega\right)}&\leq \delta^{*}\left\|\left(\mathbf{V}^{l},\rho^{l}\right)\right\|_{L^2\left(\Omega\right)}+\left\|\left(\mathbf{V}^{d},\rho^{d}\right)\right\|_{L^2\left(\Omega\right)}
\\
&\leq \delta^{*}C_{1}e^{\Lambda T^{**}}+C_{15}\left[\left(\delta^{*}\right)^{\frac{3}{2}}e^{\frac{3}{2}\Lambda T^{**}}+\left(\delta^{*}\right)^{2}e^{2\Lambda T^{**}}\right]
\\
&\leq 
\delta^{*}e^{\Lambda T^{**}}\left(C_{1}+C_{15}\left(2\varepsilon_{0}\right)^{\frac{1}{2}}+2C_{15}\varepsilon_{0}\right)<2\delta^{*}C_{1}e^{\Lambda T^{**}},
\end{aligned}
\end{align*}
which is a contradiction to \(\eqref{wuquduoxiang0104}_{2}\). Thus, \(T^{\delta^{*}}=\min\left\{T^{\delta^{*}},T^{*},T^{**}\right\}\). 

By a direct computation with \eqref{pingtai0104}, we have 
\begin{align*}
\begin{aligned}
\left\|V_{2}^{\delta^{*}}\left(T^{\delta^{*}}\right)\right\|_{L^{1}\left(\Omega\right)}&\geq \delta^{*}
\left\|V_{2}^{l}\left(T^{\delta^{*}}\right)\right\|_{L^{1}\left(\Omega\right)}-\left\|V_{2}^{d}\left(T^{\delta^{*}}\right)\right\|_{L^{1}\left(\Omega\right)}
\\
&\geq \delta^{*}e^{\Lambda T^{\delta^{*}}}m_{0}-\left|\Omega\right|^{\frac{1}{2}}
\left\|V_{2}^{d}\left(T^{\delta^{*}}\right)\right\|_{L^{2}\left(\Omega\right)}
\\
&\geq 2\varepsilon_{0}m_{0}
-\left|\Omega\right|^{\frac{1}{2}}
C_{15}\left[\left(\delta^{*}\right)^{\frac{3}{2}}e^{\frac{3\Lambda T^{\delta^{*}}}{2}}
+\left(\delta^{*}\right)^{2}e^{2\Lambda T^{\delta^{*}}}\right]
\\
&\geq 2\varepsilon_{0} (m_{0}-2\left|\Omega\right|^{\frac{1}{2}}C_{15}\left(2\varepsilon_{0}\right)^{\frac{1}{2}})\geq m_{0}\varepsilon_{0}.
\end{aligned}
\end{align*}
Similarly, one can verify the same conclusion for \(\rho^{\delta^{*}}\) and \(V_{1}^{\delta^{*}}\).
\end{proof}
\subsection{Proof of the Theorem \ref{wendingxing0225}}\label{wendingxing0226}
The section presents the detailed analysis of the linear and nonlinear stability.
\begin{proof}\textbf{The proof of Theorem \ref{wendingxing0225}.}

(1) One can refer to \cite{Choe2003} for the global well-posedness of the linearized system \eqref{xianxing1219}. Since \(h<0\) and \(\nabla\rho_{0}=h\nabla f\), we obtain 
\begin{align}\label{xianxing0226}
\begin{cases}
\frac{\rho_{t}}{-h}=\mathbf{V}\cdot\nabla f,
\\
\rho_{0}\mathbf{V}_{t}+\nabla P=\mu\Delta\mathbf{V}-\rho\nabla f,
\\
\nabla\cdot\mathbf{V}=0,~\mathbf{V}|_{\partial\Omega}=\mathbf{0}.
\end{cases}
\end{align}
Multiply \(\eqref{xianxing0226}_{1}\) and \(\eqref{xianxing0226}_{2}\) by \(\rho\) and \(\mathbf{V}\), respectively, then integrate over \(\Omega\) and add up the results, we have 
\begin{align*}
\frac{d}{dt}\int_{\Omega}\left[
\frac{\rho^2}{-h}+\rho_{0}\mathbf{V}^2
\right]dxdy+2\mu\left\|\nabla\mathbf{V}\right\|_{L^2\left(\Omega\right)}^2=0,
\end{align*}
which deduces that 
\begin{align}\label{xiawu0226}
\left\|\rho\right\|_{L^2\left(\Omega\right)}^2
+\left\|\mathbf{V}\right\|_{L^2\left(\Omega\right)}^2+\int_{0}^{t}\left\|\nabla\mathbf{V}\left(s\right)\right\|_{L^2\left(\Omega\right)}^2ds\leq C\left(\left\|\rho\left(0\right)\right\|_{L^2\left(\Omega\right)}^2+\left\|\mathbf{V}\left(0\right)\right\|_{L^2\left(\Omega\right)}^2\right),
\end{align}
where \(C=C\left(h,\rho_{0},\mu\right)>0\) is independent of \(t\).

Multiply \(\eqref{xianxing0226}_{2}\) by \(\mathbf{V}_{t}\) and integrate over \(\Omega\), then utilize the Cauchy inequality, we can obtain 
\begin{align}\label{fangqi0226}
\lim\limits_{t\rightarrow 0^{+}}
\int_{\Omega}\rho_{0}\mathbf{V}_{t}^2 dxdy 
\leq C\left(\left\|\rho\left(0\right)\right\|_{L^2\left(\Omega\right)}^2+\left\|\mathbf{V}\left(0\right)\right\|_{H^2\left(\Omega\right)}^2\right),
\end{align}
where \(C=C\left(\mu,f,\rho_{0}\right)>0\) is independent of \(t\).

Differentiating \(\eqref{xianxing0226}_{2}\) with respect to \(t\) and utilizing \(\eqref{xianxing0226}_{1}\), multiplying the result by \(\mathbf{V}_{t}\) then integrating it over \(\Omega\) give that 
\begin{align*}
\frac{d}{dt}\int_{\Omega}\left[
\rho_{0}\mathbf{V}_{t}^2-h\left|\mathbf{V}\cdot\nabla f\right|^2
\right]dxdy+2\mu\left\|\nabla\mathbf{V}_{t}\right\|_{L^2\left(\Omega\right)}^2=0,
\end{align*}
which together with \eqref{fangqi0226} yields that 
\begin{align}\label{bufangqi0226}
\left\|\mathbf{V}_{t}\right\|_{L^2\left(\Omega\right)}^2+\int_{0}^{t}\left\|\nabla\mathbf{V}_{s}\right\|_{L^2\left(\Omega\right)}^2ds
\leq C\left(\left\|\rho\left(0\right)\right\|_{L^2\left(\Omega\right)}^2+\left\|\mathbf{V}\left(0\right)\right\|_{H^2\left(\Omega\right)}^2\right),
\end{align}
where \(C=C\left(\mu,f,\rho_{0},h\right)>0\) is independent of \(t\).

Utilize \eqref{xiawu0226}, \eqref{bufangqi0226}, we apply Stokes' estimate to \(\eqref{xianxing0226}_{2}\) and get 
\begin{align}\label{guji0226}
\left\|\nabla^2\mathbf{V}\right\|_{L^2\left(\Omega\right)}^2
+\left\|\nabla P\right\|_{L^2\left(\Omega\right)}^2
\leq C\left(\left\|\rho\left(0\right)\right\|_{L^2\left(\Omega\right)}^2+\left\|\mathbf{V}\left(0\right)\right\|_{H^2\left(\Omega\right)}^2\right),
\end{align}
where \(C=C\left(\mu,f,\rho_{0},h,\Omega\right)>0\) is independent of \(t\).

From the Lemma \ref{feichangzhongyao0226}, we also have 
\begin{align}\label{jingweitianren0226}
\left\|\nabla \mathbf{V}\right\|_{L^2\left(\Omega\right)}^2
\leq C\left(\left\|\rho\left(0\right)\right\|_{L^2\left(\Omega\right)}^2+\left\|\mathbf{V}\left(0\right)\right\|_{H^2\left(\Omega\right)}^2\right),
\end{align}
where \(C=C\left(\mu,f,\rho_{0},h,\Omega\right)>0\) is independent of \(t\).

From \eqref{xiawu0226}, \eqref{bufangqi0226}, \eqref{guji0226} and 
\eqref{jingweitianren0226}, we can conclude \eqref{xuyaode0226}.

Furthermore, from \(\int_{0}^{t}
\left\|\nabla\mathbf{V}\right\|_{L^2\left(\Omega\right)}^2ds<+\infty\) for any \(t>0\) and
\(\left\|\mathbf{V}\right\|_{H^1\left(\Omega\right)}\) is continuous with respect to \(t\)\cite{Choe2003}, then the \eqref{bijin0226} holds.

(2)From the theorem 4 in \cite{Choe2003}, if \(\left\|\nabla\mathbf{V}\right\|_{L^2\left(\Omega\right)}<+\infty\) for any \(t>0\), then 
the maximum existence time of local strong solution to the nonlinear system \eqref{raodong1219} can be positive infinite.
In addition, from \cite{Choe2003}, we can conclude that 
\begin{align}\label{xu0226}
\inf\limits_{\left(x,y\right)\in\Omega}\left\{\rho+\rho_{0}\right\}=\inf\limits_{\left(x,y\right)\in\Omega}\left\{\rho\left(0\right)+\rho_{0}\right\}
,~ 
\sup\limits_{\left(x,y\right)\in\Omega}\left\{\rho+\rho_{0}\right\}=\sup\limits_{\left(x,y\right)\in\Omega}\left\{\rho\left(0\right)+\rho_{0}\right\}.
\end{align}
% which implies 
% \begin{align}\label{mashang0226}
% \left\|\rho\right\|_{L^{\infty}\left(\Omega\right)}=K+\left\|\rho_{0}\right\|_{L^{\infty}\left(\Omega\right)}.
% \end{align}
Since \(h\equiv \text{const}<0\) and \(\nabla\rho_{0}=h\nabla f\), we have 
\begin{align}\label{raodong0226}
\begin{cases}
\frac{\rho_{t}}{-h}+\frac{\mathbf{V}\cdot\nabla\rho}{-h}=\mathbf{V}\cdot\nabla f,
\\
\left(\rho+\rho_{0}\right)\mathbf{V}_{t}+\left(\rho+\rho_{0}\right)\mathbf{V}\cdot\nabla\mathbf{V}
+\nabla P=\mu\Delta\mathbf{V}-\rho\nabla f,
\\
\nabla\cdot \mathbf{V}=0,~
\mathbf{V}|_{\partial\Omega}=\mathbf{0}.
\end{cases}
\end{align}
Multiply \(\eqref{raodong0226}_{1}\) and \(\eqref{raodong0226}_{2}\) by \(\rho\) and \(\mathbf{V}\), respectively, integrate the results over \(\Omega\) and add up them, we have 
\begin{align}\label{xuyao1110226}
\frac{d}{dt}\int_{\Omega}
\left[\frac{\rho^2}{-h}+\left(\rho+\rho_{0}\right)\mathbf{V}^2\right]dxdy+2\mu\left\|\nabla\mathbf{V}\right\|_{L^2\left(\Omega\right)}^2
=0,
\end{align}
which deduces that 
\begin{align}\label{genghao0226}
\left\|\rho\right\|_{L^2\left(\Omega\right)}^2
+\left\|\mathbf{V}\right\|_{L^2\left(\Omega\right)}^2+\int_{0}^{t}\left\|\nabla\mathbf{V}\right\|_{L^2\left(\Omega\right)}^2ds
\leq C\left(\left\|\rho\left(0\right)\right\|_{L^2\left(\Omega\right)}^2+\left\|\mathbf{V}\left(0\right)\right\|_{L^2\left(\Omega\right)}^2\right),
\end{align}
where \(C=C\left(h,\mu,\sigma,K\right)>0\) is independent of \(t\). We can conclude that if \(\left\|\rho\left(0\right)\right\|_{L^2\left(\Omega\right)}^2+\left\|\mathbf{V}\left(0\right)\right\|_{H^2\left(\Omega\right)}^2\) is sufficient small, then \(\left\|\rho\right\|_{L^2\left(\Omega\right)}^2
+\left\|\mathbf{V}\right\|_{L^2\left(\Omega\right)}^2+\int_{0}^{t}\left\|\nabla\mathbf{V}\right\|_{L^2\left(\Omega\right)}^2ds\) is sufficient small for any \(t>0\).

Apply Stokes' estimate to \(\eqref{raodong0226}_{2}\), we obtain 
\begin{align}\label{jiushi0226}
\begin{aligned}
\left\|\nabla^2\mathbf{V}\right\|_{L^2\left(\Omega\right)}^2+&\left\|\nabla P\right\|_{L^2\left(\Omega\right)}^2
\leq C\left(K,\mu,\Omega,f\right)
\left[\left\|\mathbf{V}_{t}\right\|_{L^2\left(\Omega\right)}^2+\left\|\mathbf{V}\cdot\nabla \mathbf{V}\right\|_{L^2\left(\Omega\right)}^2+\left\|\rho\right\|_{L^2\left(\Omega\right)}^2\right]
\\
&\leq C\left(K,\mu,\Omega,f\right)
\left[\left\|\mathbf{V}_{t}\right\|_{L^2\left(\Omega\right)}^2+\left\|\mathbf{V}\right\|_{L^\infty\left(\Omega\right)}^2\left\|\nabla \mathbf{V}\right\|_{L^2\left(\Omega\right)}^2+\left\|\rho\right\|_{L^2\left(\Omega\right)}^2\right]
\\
&\leq C\left(K,\mu,\Omega,f\right)
\left[\left\|\mathbf{V}_{t}\right\|_{L^2\left(\Omega\right)}^2+\left\|\mathbf{V}\right\|_{L^2\left(\Omega\right)}^2\left\|\nabla^2 \mathbf{V}\right\|_{L^2\left(\Omega\right)}^2+\left\|\rho\right\|_{L^2\left(\Omega\right)}^2\right],
\end{aligned}
\end{align}
where the Lemmas \ref{zuidazhi0201} and \ref{feichangzhongyao0226} are used. Since \(\left\|\mathbf{V}\right\|_{L^2\left(\Omega\right)}^2\) is sufficient small, from \eqref{jiushi0226}, we can obtain 
\begin{align}\label{xi0226}
\left\|\nabla^2\mathbf{V}\right\|_{L^2\left(\Omega\right)}^2+\left\|\nabla P\right\|_{L^2\left(\Omega\right)}^2\leq C\left(K,\mu,\Omega,f,\hat{\delta}\right)
\left[\left\|\mathbf{V}_{t}\right\|_{L^2\left(\Omega\right)}^2+\left\|\rho\right\|_{L^2\left(\Omega\right)}^2\right],
\end{align}
where \(C\left(K,\mu,\Omega,f,\hat{\delta}\right)>0\) is independent of \(t\). 
Multiply \(\eqref{raodong0226}_{2}\) by \(\mathbf{V}_{t}\) and integrate over \(\Omega\), we obtain 
\begin{align*}
\begin{aligned}
\frac{\mu}{2}\frac{d}{dt}\left\|\nabla\mathbf{V}\right\|_{L^2\left(\Omega\right)}^2&+\int_{\Omega}\left(\rho+\rho_{0}\right)\mathbf{V}_{t}^2dxdy 
=-\int_{\Omega}\left(\rho+\rho_{0}\right)\mathbf{V}\cdot\nabla\mathbf{V}\cdot\mathbf{V}_{t}dxdy-\int_{\Omega}\rho\cdot\nabla f\cdot\mathbf{V}_{t}dxdy,
\\
&\leq C\left(K,f\right)
\left(\left\|\mathbf{V}\right\|_{L^{\infty}\left(\Omega\right)}^2\left\|\nabla\mathbf{V}\right\|_{L^2\left(\Omega\right)}^2
+\left\|\rho\right\|_{L^2\left(\Omega\right)}^2
\right)+\varepsilon\int_{\Omega}\left(\rho+\rho_{0}\right)\mathbf{V}_{t}^2dxdy,
\end{aligned}
\end{align*}
where the Cauchy inequality is used and \( C\left(K,f\right)>0\) is independent of \(t\). Thus, we have 
\begin{align*}
\begin{aligned}
\frac{d}{dt}\left\|\nabla\mathbf{V}\right\|_{L^2\left(\Omega\right)}^2+\left\|\mathbf{V}_{t}\right\|_{L^2\left(\Omega\right)}^2 \leq C\left(K,\sigma,\mu,f,\hat{\delta}\right)
\left(\left\|\mathbf{V}\right\|_{L^{2}\left(\Omega\right)}^2\left\|\nabla^2\mathbf{V}\right\|_{L^2\left(\Omega\right)}^2
+\left\|\rho\right\|_{L^2\left(\Omega\right)}^2
\right),
\end{aligned}
\end{align*}
where  \(C\left(K,\mu,\Omega,f,\hat{\delta},\sigma\right)>0\) is independent of \(t\) and
which together with \eqref{xi0226}
yields that 
\begin{align*}
\frac{d}{dt}\left\|\nabla\mathbf{V}\right\|_{L^2\left(\Omega\right)}^2
\leq C\left(\mu,\Omega,f,K,\sigma,\hat{\delta}\right)\left\|\rho\right\|_{L^2\left(\Omega\right)}^2,
\end{align*}
where \(C\left(K,\mu,\Omega,f,\hat{\delta},\sigma\right)>0\) is independent of \(t\) and 
which implies that 
\[
\left\|\nabla\mathbf{V}\right\|_{L^2\left(\Omega\right)}^2<+\infty,~\text{for~any~}t>0.
\]
Thus, the global existence of the strong solution to the nonlinear system \eqref{raodong1219} is verified. And from \eqref{genghao0226}, we can obtain 
\begin{align}\label{jiasan0310}
\left\|\mathbf{V}\right\|_{H^{1}\left(\Omega\right)}\rightarrow 0,~\text{as~}t\rightarrow+\infty.
\end{align}
Without loss of generality, we can assume that \(\left\|\mathbf{V}\right\|_{H^{1}\left(\Omega\right)}\) is sufficient small for any \(t>0\).

Multiply \(\eqref{raodong0226}_{2}\) by \(\mathbf{V}_{t}\) and integrate over \(\Omega\), then utilize the Cauchy inequality, we can obtain 
\begin{align}\label{fangqi110226}
\lim\limits_{t\rightarrow 0^{+}}
\int_{\Omega}\left(\rho_{0}+\rho\right)\mathbf{V}_{t}^2 dxdy 
\leq C\left(\left\|\rho\left(0\right)\right\|_{L^2\left(\Omega\right)}^2+\left\|\mathbf{V}\left(0\right)\right\|_{H^2\left(\Omega\right)}^2\right),
\end{align}
where \(C=C\left(\mu,f,\rho_{0},K\right)>0\) is independent of \(t\).

Differentiate \(\eqref{raodong0226}_{2}\) with respect to \(t\), multiply this result by \(\mathbf{V}_{t}\) and integrate it over \(\Omega\), we obtain 
\begin{align}\label{xuyao02}
\frac{d}{dt}\int_{\Omega}
\left[\left(\rho+\rho_{0}\right)\mathbf{V}_{t}^2-h\left|\mathbf{V}\cdot\nabla f\right|^2
\right]dxdy+2\mu\left\|\nabla\mathbf{V}_{t}\right\|_{L^2\left(\Omega\right)}^2=\sum\limits_{i=1}^{4}\hat{K}_{i},
\end{align}
where 
\begin{align*}
\begin{aligned}
&\hat{K}_{1}=-4\int_{\Omega}\left(\rho+\rho_{0}\right)\mathbf{V}\cdot\nabla\mathbf{V}_{t}\cdot\mathbf{V}_{t}dxdy,~\hat{K}_{2}=-2\int_{\Omega}\left(\rho+\rho_{0}\right)\mathbf{V}\cdot\nabla\left(\mathbf{V}\cdot\nabla\mathbf{V}\cdot\mathbf{V}_{t}\right)dxdy,
\\
&~~~~~~~~~~\hat{K}_{3}=-2\int_{\Omega}\left(\rho+\rho_{0}\right)\mathbf{V}_{t}\cdot\nabla\mathbf{V}\cdot\mathbf{V}_{t}dxdy,~
\hat{K}_{4}=-2\int_{\Omega}\rho\mathbf{V}\cdot\nabla\left(\nabla f\cdot\mathbf{V}_{t}\right)dxdy.
\end{aligned}
\end{align*}
Then a direct computation gives that 
\begin{align*}
&\hat{K}_{1},\hat{K}_{3}\leq C\left(K,\Omega\right)\left\|\nabla\mathbf{V}\right\|_{L^{2}\left(\Omega\right)}\left\|\nabla\mathbf{V}_{t}\right\|_{L^2\left(\Omega\right)}^2,
\\
&\hat{K}_{2}\leq C\left(K,\Omega\right)
\left\|\nabla\mathbf{V}\right\|_{L^2\left(\Omega\right)}^4\left\|\nabla^2\mathbf{V}\right\|_{L^2\left(\Omega\right)}^2+\varepsilon\left\|\nabla\mathbf{V}_{t}\right\|_{L^2\left(\Omega\right)}^2,
\\
&\hat{K}_{4}\leq C\left(f,\Omega\right)
\left\|\rho\right\|_{L^2\left(\Omega\right)}^{\frac{4}{3}}\left\|\mathbf{V}\right\|_{L^2\left(\Omega\right)}^2
\left\|\nabla^2\mathbf{V}\right\|_{L^2\left(\Omega\right)}^2+\varepsilon\left\|\nabla\mathbf{V}_{t}\right\|_{L^2\left(\Omega\right)}^2.
\end{align*}
Thus, from \eqref{xuyao02} and \(\left\|\mathbf{V}\right\|_{H^{1}\left(\Omega\right)}\) is sufficient small, we have
\begin{align*}
\begin{aligned}
\frac{d}{dt}\int_{\Omega}
\left[\left(\rho+\rho_{0}\right)\mathbf{V}_{t}^2-h\left|\mathbf{V}\cdot\nabla f\right|^2
\right]dxdy+\left\|\nabla\mathbf{V}_{t}\right\|_{L^2\left(\Omega\right)}^2
&\leq C
\left\|\rho\right\|_{L^2\left(\Omega\right)}^{\frac{4}{3}}\left\|\mathbf{V}\right\|_{L^2\left(\Omega\right)}^2
\left\|\nabla^2\mathbf{V}\right\|_{L^2\left(\Omega\right)}^2
\\
&+
C\left\|\nabla\mathbf{V}\right\|_{L^2\left(\Omega\right)}^4\left\|\nabla^2\mathbf{V}\right\|_{L^2\left(\Omega\right)}^2,
\end{aligned}
\end{align*}
where \(C=C\left(\mu,\Omega,f,K,\hat{\delta}\right)>0\) is independent of t and
which together with \eqref{xi0226} yields that
\begin{align*}
\frac{d}{dt}\int_{\Omega}
\left[\left(\rho+\rho_{0}\right)\mathbf{V}_{t}^2-h\left|\mathbf{V}\cdot\nabla f\right|^2
\right]dxdy&+\left\|\nabla\mathbf{V}_{t}\right\|_{L^2\left(\Omega\right)}^2
\leq 
C\bigg{[}\left\|\rho\right\|_{L^2\left(\Omega\right)}^{\frac{4}{3}}\left\|\mathbf{V}\right\|_{L^2\left(\Omega\right)}^2\left\|\mathbf{V}_{t}\right\|_{L^2\left(\Omega\right)}^2
\\
&+\left\|\rho\right\|_{L^2\left(\Omega\right)}^{\frac{10}{3}}\left\|\mathbf{V}\right\|_{L^2\left(\Omega\right)}^2+\left\|\nabla\mathbf{V}\right\|_{L^2\left(\Omega\right)}^2\left\|\mathbf{V}_{t}\right\|_{L^2\left(\Omega\right)}^2
\bigg{]},
\end{align*}
where \(C=C\left(\mu,\Omega,K,f,\hat{\delta}\right)>0.\)
Since \(\left\|\rho\right\|_{L^2\left(\Omega\right)}\) and \(\left\|\mathbf{V}\right\|_{H^{1}\left(\Omega\right)}\) are sufficient small, we have 
\[
\frac{d}{dt}\int_{\Omega}
\left[\left(\rho+\rho_{0}\right)\mathbf{V}_{t}^2-h\left|\mathbf{V}\cdot\nabla f\right|^2
\right]dxdy+\left\|\nabla\mathbf{V}_{t}\right\|_{L^2\left(\Omega\right)}^2
\leq C\left\|\nabla\mathbf{V}\right\|_{L^2\left(\Omega\right)}^2,
\]
where \(C=C\left(\mu,\Omega,K,f,\hat{\delta},\sigma\right)>0\) is independent of \(t\). From \eqref{genghao0226} and \eqref{fangqi110226}, we obtain 
\begin{align}\label{jiaer0310}
\left\|\mathbf{V}_{t}\right\|_{L^2\left(\Omega\right)}^2+\int_{0}^{t}\left\|\nabla\mathbf{V}_{s}\right\|_{L^2\left(\Omega\right)}^2ds
\leq C\left(\left\|\rho\left(0\right)\right\|_{L^2\left(\Omega\right)}^2+\left\|\mathbf{V}\left(0\right)\right\|_{H^2\left(\Omega\right)}^2\right),
\end{align}
where \(C=C\left(\mu,\Omega,K,f,h,\hat{\delta},\sigma,\rho_{0}\right)>0\) is independent of \(t\).

Then, from \eqref{genghao0226} and \eqref{xi0226}, we obtain 
\begin{align}\label{jiayi0310}
\left\|\nabla^2\mathbf{V}\right\|_{L^{2}\left(\Omega\right)}^2+\left\|\nabla P\right\|_{L^2\left(\Omega\right)}^2
\leq C\left(\left\|\rho\left(0\right)\right\|_{L^2\left(\Omega\right)}^2+\left\|\mathbf{V}\left(0\right)\right\|_{H^2\left(\Omega\right)}^2\right),
\end{align}
where \(C=C\left(\mu,\Omega,K,f,h,\hat{\delta},\sigma,\rho_{0}\right)>0\) is independent of \(t\). Thus, from \eqref{genghao0226}, \eqref{jiasan0310}, \eqref{jiaer0310} and \eqref{jiayi0310}, we can verify the conclusions desired.
\end{proof}

\section{Appendix}\label{appendix0304}
\begin{proof}\textbf{The proof of Lemma \ref{wentaijie0201}.}

Through a straightforward calculation using \(\eqref{wentifangcheng0201}_{1}\) and \(\eqref{wentifangcheng0201}_{3}\), we have
\begin{align*}
\begin{aligned}
&\int_{\Omega}\rho\mathbf{V}\cdot\nabla\mathbf{V}\cdot\mathbf{V}dxdy=-2\int_{\Omega}\nabla\rho\cdot\mathbf{V}\mathbf{V}^2dxdy=0,
\\
&\int_{\Omega}\rho\nabla f\cdot \mathbf{V}dxdy=-\int_{\Omega}\nabla\rho\cdot\mathbf{V}fdxdy=0.
\end{aligned}
\end{align*}
Next, multiply \(\eqref{wentifangcheng0201}_{2}\) by \(\mathbf{V}\) and integrate the result over \(\Omega\), we get 
\begin{align*}
\int_{\Omega}\left|\nabla\mathbf{V}\right|^2dxdy=0,
\end{align*}
which together with \(\mathbf{V}|_{\partial\Omega}=\mathbf{0}\) yields that 
\(\mathbf{V}\equiv \mathbf{0}\), that is, \(\mathbf{u}\equiv \mathbf{0}\). 

From \(\eqref{wentifangcheng0201}_{2}\), one can infer that 
\begin{align*}
\nabla P_{0}=-\rho_{0}\nabla f,
\end{align*}
which indicates that 
\(
\left(\partial_{y}\rho_{0},-\partial_{x}\rho_{0}\right)\cdot 
\left(\partial_{x}f,\partial_{y}f\right)=0.
\) Thus, \eqref{wutian0201} holds.
\end{proof}

\begin{proof}\textbf{The proof of Lemma \ref{poincarebudengshi0201}.}

First, since \(\mathbf{V}=\left(V_{1},V_{2}\right)\in \left[H_{0}^{1}\left(\Omega\right)\right]^2\) and \(\Omega\) is bounded and smooth, we can find a bounded rectangular region \(\widetilde{\Omega}=\left[-L_{1},L_{1}\right]\times \left[-L_{2},L_{2}\right]\) such that 
\(\Omega\subset \widetilde{\Omega}\) and \(\mathbf{V}\equiv\mathbf{0}\) in \(\widetilde{\Omega}/\Omega\). we only consider \(V_{1}\), as the proof for \(V_{2}\) is similar. we first
assume \(\mathbf{V}\) is smooth due to the density.

From the Newton-Leibniz formula, we obtain 
\begin{align*}
V_{1}\left(x,y\right)=\int_{-L_{1}}^{x}
\partial_{s}V_{1}\left(s,y\right)ds,~V_{1}\left(x,y\right)=\int_{-L_{2}}^{y}
\partial_{s}V_{1}\left(x,s\right)ds,
\end{align*}
which together with H\'{o}lder inequality yields that 
\begin{align*}
\left\|V_{1}\right\|_{L^2\left(\Omega\right)}=\left\|V_{1}\right\|_{L^2\left(\widetilde{\Omega}\right)}
\leq C\left(\Omega\right)\left\|\nabla V_{1}\right\|_{L^2\left(\Omega\right)}.
\end{align*}

By the mathematical induction, we assume that we have 
\begin{align}\label{xuyao0204}
\left\|V_{1}\right\|_{L^{2k}\left(\Omega\right)}\leq C\left(k,\Omega\right)\left\|\nabla V_{1}\right\|_{L^2\left(\Omega\right)}
\end{align}
Similarly, by Newton-Leibniz formula, we have 
\begin{align*}
&\left|V_{1}\left(x,y\right)\right|^{k+1}\leq (k+1)\int_{-L_{1}}^{L_{1}}
\left|\partial_{x}V_{1}\left(x,y\right)\right|\left|V_{1}\left(x,y\right)\right|^{k}dx,
\\
&\left|V_{1}\left(x,y\right)\right|^{k+1}\leq \left(k+1\right)\int_{-L_{2}}^{L_{2}}
\left|\partial_{y}V_{1}\left(x,y\right)\right|\left|V_{1}\left(x,y\right)\right|^{k}dy.
\end{align*}
Multiply the above two inequalities and integrate the result over \(\widetilde{\Omega}\), we get 
\begin{align*}
\left\|V_{1}\right\|_{L^{2\left(k+1\right)}\left(\Omega\right)}^{2(k+1)}
\leq C\left(\Omega\right)
\left\|\nabla V_{1}\right\|_{L^2\left(\Omega\right)}^{2\left(k+1\right)},
\end{align*}
where H\'{o}lder inequality and \eqref{xuyao0204} are utilized. Thus, the proof is completed.
\end{proof}

\begin{proof}\textbf{The proof of Lemma \ref{feichangzhongyao0226}.}

Utilize the integrate by parts and the H\"{o}lder inequality, we have 
\[
\left\|\nabla v\right\|_{L^2\left(\Omega\right)}^2=\int_{\Omega}\nabla v\cdot\nabla v dxdy 
=-\int_{\Omega}v\nabla^2 v dxdy
\leq \left\|v\right\|_{L^2\left(\Omega\right)}
\left\|\nabla^2 v\right\|_{L^2\left(\Omega\right)}.
\]
\end{proof}

\begin{proof}\textbf{The proof of Lemma \ref{budengshidezhengming0201}.}
    Let \(H\left(t\right)=\int_{0}^{t}
    \left\|\nabla\mathbf{V}\right\|_{L^2\left(\Omega\right)}^6ds
    \). Then by the Cauchy inequality, we have 
    \[
      H'\left(t\right)\leq C\left[H\left(t\right)\right]^3+\widetilde{C}_{1}(t+1)^3,
    \]
   where \(|\widetilde{C}_{1}-C_{1}|\) is small enough. Then, we obtain  
    \[
    \frac{d}{dt}\left[e^{-C\int_{0}^{t}\left[H\left(s\right)\right]^2ds}H\left(t\right)\right]\leq \widetilde{C}_{1}(t+1)^3.   
    \]
    which leads to
    \[
      e^{-C\int_{0}^{t}\left[H\left(s\right)\right]^2ds}H\left(t\right)\leq \widetilde{C}_{1}(t+1)^4,
    \]
    and further
    \[
      e^{-C\int_{0}^{t}\left[H\left(s\right)\right]^2ds}H^2\left(t\right)\leq \widetilde{C}_{1}^2(t+1)^8. 
    \]
    Consequently, 
    \[
      -\frac{d}{dt}e^{-C\int_{0}^{t}\left[H\left(s\right)\right]^2ds}\leq \widetilde{C}_{1}^2 (t+1)^8,
    \]
    As a result, when \(1-\widetilde{C}_{1}^2(t+1)^9>0\), we obtain
    \[
      e^{C\int_{0}^{t}\left[H\left(s\right)\right]^2ds}\leq \frac{1}{1-\widetilde{C}_{1}^2(t+1)^9},
    \]
    which implies the existence of \(T^{*}\) and \(C^{*}\). Moreover,  \(H\left(t\right)\leq C^{*}\), indicating that 
    \(\left\|\nabla\mathbf{V}\left(t\right)\right\|_{L^2\left(\Omega\right)}^2
    \leq C^{*}\).
  \end{proof}

\section*{Acknowledgments}
Acknowledgments. L. Li was supported by the Young Scientists Fund of the 
National Natural Science Foundation of China (No. 12301131).
% \begin{remark}\label{lingyizhong0206}For the scenario where \(h>0\) everywhere, we just adjust some formulas in the above proof. For instance, we replace \(\left(\mathbf{u}_{0},\theta_{0}\right)\) by \(\left(\widetilde{\widetilde{\mathbf{u}}}_{0},\widetilde{\widetilde{\theta}}_{0}\right)\) found in Proposition \ref{qujizhi0206} and Remark \ref{bijiao0206}.
% We consider the following equations similar as \eqref{suizaiyi0104}
% \begin{align}\label{suizaiyi0206}
% \begin{cases}
% \frac{\rho_{t}^{d}}{h}+\frac{\mathbf{V}^{d}\cdot\nabla\rho_{0}}{h}=-\frac{\mathbf{V}^{\delta^{*}}\cdot\nabla\rho^{\delta^{*}}}{h},
% \\
% \rho_{0}\mathbf{V}_{t}^{d}+\nabla P^{d}-\mu\Delta\mathbf{V}^{d}+\rho^{d}\nabla f=-\rho^{\delta^{*}}\mathbf{V}_{t}^{\delta^{*}}-\left(\rho^{\delta^{*}}+\rho_{0}\right)\mathbf{V}^{\delta^{*}}\cdot\nabla\mathbf{V}^{\delta^{*}},
% \\
% \nabla\cdot\mathbf{V}^{d}=0,~
% \mathbf{V}^{d}|_{\partial\Omega}=\mathbf{0},~\left(\mathbf{V}^{d},\rho^{d}\right)\left(0\right)=\left(\mathbf{0},0\right),
% \end{cases}
% \end{align}
% and from the inequality in Remark \ref{bijiao0206} we obtain the same inequality as \eqref{pingtai0104}.
% \begin{align}\label{pingtai0206}
% \int_{\Omega}\left[\left|\frac{\rho^{d}}{\sqrt{h}}\right|^2+\rho_{0}\left|\mathbf{V}^{d}\right|^2\right]dxdy\leq C_{6}\left(\delta^{*}\right)^3 e^{3\Lambda t}.
% \end{align} 
% The remaining steps can be followed in the same manner to obtain the conclusion desired.
% \end{remark}

\end{document}